\documentclass[11pt,twoside]{amsart}
\usepackage{amsmath,amssymb,amsthm}

\newtheorem{thm}{Theorem}[section]

 \newtheorem{cor}{Corollary}[section]
 \newtheorem{lem}{Lemma}[section]
 \newtheorem{prop}{Proposition}[section]
 \newtheorem{defn}{Definition}[section]
\newtheorem{rem}{Remark}[section]

\def\Id{{\rm Id}\,}
\def\d{\partial}
\def\ddj{\dot \Delta_j}

\def\tilde{\widetilde}
\def\hat{\widehat}

\def\wt{\widetilde}

\newcommand\R{\mathbb{R}}

\newcommand\Z{\mathbb{Z}}

\newcommand{\ep}{\varepsilon}

\newcommand{\Supp}{\hbox{Supp}\,}
\newcommand{\Tr}{\hbox{\rm{Tr}\,}}

\renewcommand{\div}{\mbox{\rm div}\;\!}
\newcommand{\divx}{\mbox{\rm div}_x\,}

\def\cA{{\mathcal A}}

\def\cC{{\mathcal C}}
\def\cD{{\mathcal D}}

\def\cF{{\mathcal F}}

\def\cL{{\mathcal L}}

\def\cP{{\mathcal P}}

\def\cT{{\mathcal T}}
\def\cX{{\mathcal X}}

\begin{document}
\title[Decay estimates for the compressible Navier-Stokes-Fourier equations]{
Optimal decay estimates in the critical $L^{p}$ framework for flows of compressible viscous and heat-conductive gases}
\author{Rapha\"el Danchin}
\address{Universit\'{e} Paris-Est,  LAMA (UMR 8050), UPEMLV, UPEC, CNRS, 61 avenue du G\'en\'eral de Gaulle, 94010 Cr\'eteil Cedex 10}
\email{danchin@univ-paris12.fr}
\thanks{The first author is supported by ANR-15-CE40-0011 and by the  Institut Universitaire de France.}
\author{Jiang Xu}
\address{Department of Mathematics,  Nanjing
University of Aeronautics and Astronautics,
Nanjing 211106, P.R.China,}
\email{jiangxu\underline{ }79math@yahoo.com}
\thanks{The second author is partially supported by the National
Natural Science Foundation of China (11471158), the Program for New Century Excellent
Talents in University (NCET-13-0857) and the Fundamental Research Funds for the Central
Universities (NE2015005). He would like to thank Professor A. Matsumura for  introducing  him
to  the decay problem for partially parabolic equations when he visited Osaka University.  He is also grateful to
Professor R. Danchin for his kind hospitality when visiting the LAMA in UPEC}


\begin{abstract}
The global existence issue in critical regularity spaces for the full Navier-Stokes equations
satisfied by compressible viscous and heat-conductive gases has been first addressed in \cite{D2}, then recently extended to the  general $L^p$ framework in \cite{DH}.
In the present work, we establish decay estimates for the global solutions constructed in \cite{DH}, under an
additional mild  integrability  assumption  that is satisfied if the low frequencies of the initial data are in  $L^{r}(\R^d)$ with $\frac p2\leq r\leq \min\{2,\frac d2\}\cdotp$
As a by-product in the case $d=3,$  we recover     the classical  decay
 rate   $t^{-\frac34}$ for $t\to+\infty$ that has been  observed by A. Matsumura and T. Nishida in \cite{MN2}  for  solutions with high  Sobolev regularity.  Compared to  a recent paper of us \cite{DX} dedicated to the  barotropic case,
 not only we are able to treat the full system, but we also weaken the low frequency
 assumption and  improve the decay exponents for the high
 frequencies of the solution.
\end{abstract}

\maketitle

\section{Introduction}\setcounter{equation}{0}
The motion of general viscous and heat conductive gases  is governed by
\begin{equation}\label{R-E1}
\left\{\begin{array}{l}
\partial_t\rho+\divx(\rho u)=0,\\ [1mm]
\partial_t(\rho u)+\divx(\rho u\otimes u)+\nabla_x P=\divx \tau,\\ [1mm]
\partial_t\Big[\rho\Big(\frac{|u|^2}{2}+e\Big)\Big]+\divx\Big[u\Big(\rho\Big(\frac{|u|^2}{2}+e\Big)+P\Big)\Big]=\divx(\tau\cdot u-q),
\end{array}\right.
\end{equation}
where $\rho=\rho(t,x)\in \mathbb{R}_{+}$ denotes the density, $u=u(t,x)\in \mathbb{R}^d$, the velocity field and $e=e(t,x)\in \mathbb{R}_{+}$, the internal energy per unit mass.  We restrict ourselves to the case of a Newtonian fluid: the viscous stress tensor is $\tau=\lambda\,\divx u \,\Id+2\mu D_x(u)$, where $D_x(u)\triangleq\frac12(D_x u+{}^T\!D_x u)$ stands for  the  deformation tensor, and $\divx$ is the divergence operator with respect to the spatial variable. The  bulk and shear viscosities are
 supposed to satisfy
\begin{equation}\label{eq:posvis}\mu > 0\quad\hbox{and}\quad \nu\triangleq\lambda+2\mu>0.\end{equation}
The heat flux $q$ is given by $ q=-\kappa\nabla_x \mathcal{T}$ where  $\kappa>0,$   and $\mathcal{T}$ stands for the temperature.
 For simplicity,  the coefficients $\lambda,$ $\mu$ and $\kappa$  are taken constant in all that follows.
 \smallbreak
   It is well known that  combining the second and third equations of \eqref{R-E1}  yields
$$
\d_t(\rho e)+\divx(\rho u e)+P\divx u-\kappa\Delta_{\!x}\cT=2\mu D_x(u):D_x(u)+\lambda(\divx u)^2.
$$
In order to reformulate System \eqref{R-E1}  in terms of $\rho,$ $u$ and
$\cT$ only, we make the additional assumption that
 the internal energy $e=e(\rho,\cT)$ satisfies  Joule law :
\begin{equation}\label{eq:Joule}\d_\cT e=C_v\ \hbox{ for some  positive constant } C_{v}\end{equation}
and that the pressure function $P=P(\rho,\cT)$ is  of the form
\begin{equation}\label{eq:pressure}
P(\rho,\mathcal{T})=\pi_{0}(\rho)+\mathcal{T}\pi_{1}(\rho),\end{equation}
where $\pi_{0}$ and $\pi_{1}$ are given smooth functions\footnote{One can thus consider  e.g. perfect gases ($\pi_{0}(\rho)=0$ and  $\pi_{1}(\rho)=R\rho$ with $R>0$) or Van-der-Waals  fluids ($\pi_{0}(\rho)=-\alpha\rho^2$, $\pi_{1}(\rho)=\beta\rho/(\delta-\rho)$ with $\alpha,\beta,\delta>0$).}.
Then taking advantage of the
 Gibbs relations for the internal energy and the Helmholtz free energy, we get  the  Maxwell relation
$$\rho^2\d_\rho e(\rho,\cT)=P(\rho,\cT)-\cT\d_\cT P(\rho,\cT)=\pi_0(\rho),$$
and end up with  the following temperature equation:
\begin{equation}\label{R-E8}
\rho C_v(\d_t\cT+u\cdot\nabla_x\cT)
+ \cT\pi_1(\rho)\divx u-\kappa\Delta_x\cT=2\mu D_x(u):D_x(u)+\lambda(\divx u)^2.
\end{equation}

  We are concerned with the large time behaviour of (strong) global solutions
  to  \eqref{R-E1} in the case where   the fluid domain is the whole space $\R^d$ with $d\geq3$.
We focus on solutions  that  are close  to some
constant  equilibrium $(\bar{\rho},0,\bar{\mathcal{T}})$ with $\bar\rho>0$ and $\bar\cT>0$
satisfying the  linear stability condition:
\begin{eqnarray}\label{R-E002}\partial_{\rho}P(\bar{\rho},\bar{\mathcal{T}})>0\ \ \ \mbox{and}\ \ \ \partial_{\mathcal{T}}P(\bar{\rho},\bar{\mathcal{T}})>0.\end{eqnarray}
Recall that, starting with  the pioneering work by Matsumura and Nishida \cite{MN2},
 a number of papers have been dedicated to that issue in the case of
  solutions with high Sobolev regularity.
We here  aim at performing the long time asymptotics within the so-called \emph{critical regularity} framework,
that  is in functional spaces endowed with norms that are invariant for all $\ell>0$ by the transform:
\begin{equation}\label{R-E3}
\rho(t,x)\leadsto \rho(\ell ^2t,\ell x),\quad
u(t,x)\leadsto \ell u(\ell^2t,\ell x), \quad  \mathcal{T}(t,x) \leadsto \ell^2\mathcal{T}(\ell^2t,\ell x)  \qquad\ell>0.
\end{equation}
That definition of criticality corresponds to the scaling invariance (up to a
suitable change of the  pressure terms) of  System  \eqref{R-E1} written in terms of  $(\rho,u,\cT).$
\medbreak
Scaling invariance plays  a fundamental role in the study of evolutionary PDEs.
Recall that in the context of the incompressible Navier-Stokes equations, working in critical spaces
goes back  to the work by  Fujita \& Kato in \cite{FK} (see also more recent results in Kozono \& Yamazaki \cite{KY}
and Cannone  \cite{C}), and that it has been extended to the compressible
Navier-Stokes equations in e.g.   \cite{CD,CD2,CMZ2,D0,D2,D3,DH,H2}.

Even though, rigorously speaking, System  \eqref{R-E1} does not possess any scaling invariance,
  it is possible to solve it locally in time in Banach spaces endowed with norms
  having the  invariance  given by \eqref{R-E1}. For example, it has been proved
  in \cite{CD,D3} that  in dimension $d\geq3,$ it is  well-posed
in  $\dot{B}^{\frac{d}{p}}_{p,1}\times(\dot{B}^{\frac{d}{p}-1}_{p,1})^{d}\times\dot{B}^{\frac{d}{p}-2}_{p,1}$
 if $1\leq p<d.$  A key fact for proving that result
is    to observe that the coupling between the  equations is  low order.

As regards  the global well-posedness issue however,  the low order terms have to be taken into
account in the choice of a suitable functional framework, and it is suitable
to use    ``hybrid" Besov norms with
 different regularity indices  in the low and high frequencies.
The basic heuristics is that, for low frequencies, the first order terms predominate so that \eqref{R-E1} can  be handled by means
of  hyperbolic methods (in particular it is natural to work at the same level of regularity for the density, velocity and temperature). In contrast, in the high frequency regime, two types of behaviours coexist: the parabolic one for the velocity and temperature, and the damped one for the density.  This heuristics may be translated in terms of a priori
estimates by means of an energy method that can be  directly implemented on \eqref{R-E1}, after
spectral localization (the main difficulty arising  from the convection term in the density equation
can be by-passed thanks to suitable integration by parts and commutator estimates, see \cite{D2}).
Recently, the first author and L. He \cite{DH} extended the results of \cite{D2} to more general $L^p$ Besov spaces, and got the following statement:
 \begin{thm}\label{thm1.1}
Let $\bar{\rho}>0$ and $\bar{\mathcal{T}}$ be two constant such that \eqref{R-E002} is fulfilled. Suppose that $d\geq 3$, and that  $p$ satisfies
\begin{equation}\label{eq:p}
2\leq p<d\quad\hbox{and}\quad  p\leq 2d/(d-2).\end{equation}
There exists a constant $c=c(p,d,\lambda,\mu,P, \kappa, C_v,\bar{\rho}, \bar{\mathcal{T}})$ such that
if    $a_0\triangleq \rho_0-\bar{\rho}$ is in $\dot B^{\frac {d}{p}}_{p,1},$
 if $\upsilon_0\triangleq u_0$ is in $\dot B^{\frac {d}{p}-1}_{p,1}$, if $\theta_0\triangleq \mathcal{T}_0-\bar{\mathcal{T}}$  is in $\dot B^{\frac {d}{p}-2}_{p,1}$
 and if in addition  the low frequency part\footnote{See the definitions in
 \eqref{eq:lh} and just below \eqref{R-E15}.}  $(a_0^\ell,\upsilon_0^\ell, \theta_0^\ell)$ of
   $(a_0,\upsilon_0, \theta_0)$ is in $\dot B^{\frac d2-1}_{2,1}$  with
    \begin{eqnarray}\label{R-E4}
\cX_{p,0}\triangleq \|(a_0,\upsilon_0,\theta_0)\|^\ell_{\dot B^{\frac d2-1}_{2,1}}+\|a_0\|^h_{\dot B^{\frac dp}_{p,1}}
+\|\upsilon_0\|^h_{\dot B^{\frac d{p}-1}_{p,1}}+\|\theta_0\|^h_{\dot B^{\frac d{p}-2}_{p,1}}\leq c
\end{eqnarray}
then System \eqref{R-E1}-\eqref{R-E8} supplemented with   the initial condition
\begin{equation}\label{R-E2}
(\rho,u,\mathcal{T})|_{t=0}=(\rho_0,u_0,\mathcal{T}_{0})
\end{equation}
admits a unique global-in-time  solution $(\rho,u,\mathcal{T})$  with  $\rho=\bar{\rho}+a$, $u=\upsilon$ and $\mathcal{T}=\bar{\mathcal{T}}+\theta$, where $(a,\upsilon,\theta)$ belongs to the space $X_p$ defined by:
$$\displaylines{
(a,\upsilon,\theta)^\ell\in \wt\cC_b(\R_+;\dot B^{\frac d2-1}_{2,1})\cap  L^1(\R_+;\dot B^{\frac d2+1}_{2,1}),\quad
a^h\in \wt\cC_b(\R_+;\dot B^{\frac dp}_{p,1})\cap L^1(\R_+;\dot B^{\frac dp}_{p,1}),
\cr \upsilon^h\in  \wt\cC_b(\R_+;\dot B^{\frac dp-1}_{p,1})
\cap L^1(\R_+;\dot B^{\frac dp+1}_{p,1}),\ \ \theta^h\in  \wt\cC_b(\R_+;\dot B^{\frac dp-2}_{p,1})
\cap L^1(\R_+;\dot B^{\frac dp}_{p,1}).}
$$
Moreover, we have for some constant $C=C(p,d,\lambda,\mu,P, \kappa, C_v,\bar{\rho}, \bar{\mathcal{T}})$,
\begin{eqnarray}\label{R-E5}
\cX_{p}(t)\leq C\cX_{p,0},
\end{eqnarray}
for any $t>0$, where
\begin{multline}\label{R-E6}
 \cX_{p}(t)\triangleq\|(a,\upsilon,\theta)\|^{\ell}_{\wt L^\infty_t(\dot B^{\frac d2-1}_{2,1})}+\|(a,\upsilon,\theta)\|^{\ell}_{L^1_t(\dot B^{\frac d2+1}_{2,1})}
\\+\|a\|^{h}_{\wt L^\infty_t(\dot B^{\frac dp}_{p,1})\cap L^1_t(\dot B^{\frac dp}_{p,1})}
+\|\upsilon\|^{h}_{\wt L^\infty_t(\dot B^{\frac dp-1}_{p,1})\cap L^1_t(\dot B^{\frac dp+1}_{p,1})}
+\|\theta\|^{h}_{\wt L^\infty_t(\dot B^{\frac dp-2}_{p,1})\cap L^1_t(\dot B^{\frac dp}_{p,1})}.
\end{multline}
\end{thm}
The natural next step is to look for a more accurate description of the long time
behavior of the solutions.  Recall that
Matsumura and Nishida in \cite{MN1}
 proved that if the  initial data
 are a small perturbation in $H^3(\mathbb{R}^3)\times L^1(\mathbb{R}^3)$
of $(\bar{\rho},0,\bar{\mathcal{T}})$ then\footnote{Similar decay rates have been established in the half-space or exterior domain cases, see for example \cite{KK2,KS,MN3}.}
\begin{eqnarray}\label{R-E7}
\sup_{t\geq0}\: \langle t\rangle^{\frac{3}{4}} \|(\rho-\bar{\rho},u,\mathcal{T}-\bar{\mathcal{T}})(t)\|_{L^2}
<\infty,\quad\hbox{with }\
\langle t\rangle\triangleq \sqrt{1+t^2}. \end{eqnarray}
It turns out that the above behavior is kind of universal: some years latter
Kawashima in  \cite{Ka} exhibited
similar decay rates for   hyperbolic-parabolic composite systems satisfying what is  now called the ``Shizuta-Kawashima" stability criterion.

Still in the framework of solutions with high Sobolev regularity, there are lots of recent
improvements concerning the large time description of solutions to
the compressible Navier-Stokes equations. In particular, some informations
are now available on the wave aspect of the solutions.
In one dimension space  and in the isentropic case, Zeng \cite{Z}
 showed the $L^1$ convergence to the nonlinear Burgers' diffusive wave.
 For multi-dimensional diffusion waves, Hoff and Zumbrun \cite{HZ1,HZ2} gave a detailed analysis for the Green's function and derived the $L^\infty$ time-decay rates of diffusive waves. In \cite{LW}, Liu and Wang exhibited
   pointwise estimates of  diffusion waves with the optimal time-decay rate in odd dimension, that
   corresponds to the \emph{weak  Huygens' principle}. This was generalized later to the full system \eqref{R-E1} in \cite{L}.
 Recently, Liu \& Noh \cite{LN} provided an exhaustive classification of the different type
 of waves in the long-time asymptotics  as a combination of
  low frequency waves, the dissipative Huygens, diffusion and Riesz waves.

In the present paper, we aim at proving  optimal time-decay estimates for \eqref{R-E1}
\emph{within} the critical regularity
framework of Theorem \ref{thm1.1}.   Recall that in the (simpler)  barotropic case,  Okita \cite{O}
proved the optimal  $L^2$ decay rate in dimension $d\geq3$. The first author \cite{D8} proposed another description which allows to deal with the case  $d=2$ in the $L^2$ critical framework. Very recently, in  \cite{DX} we improved the approach in \cite{D8}
and succeeded in  establishing  optimal decay estimates in the general $L^p$ critical framework for  all dimensions $d\geq2$.
As a first attempt of generalization, we here establish similar  results
for the full Navier-Stokes equations  \eqref{R-E1}. In fact, thanks  to an improvement
of our method, we shall obtain what we believe to be the optimal decay
exponents for the full nonlinear system.


 \section{Reformulation of our problem and main results}\setcounter{equation}{0}

Let us assume that the density and the temperature tend to some positive constants $\bar{\rho}$ and $\mathcal{\bar{T}},$
 at infinity. Setting $\cA\triangleq\mu\Delta+(\lambda+\mu)\nabla\div,$
$\rho=\bar\rho(1+b)$ and $\cT=\bar\cT +{\mathfrak T},$ we see
from \eqref{R-E1} and \eqref{R-E8} that, whenever $b>-1,$ the triplet $(b,u,\theta)$ satisfies\footnote{For better readability, from now on, we
just denote $D_x$ by $D,$  $\divx$ by $\div,$ and so on.}
$$\left\{
\begin{aligned}
&\d_tb+u\cdot\nabla b+(1+b)\div u=0,\\
&\d_tu+u\cdot\nabla u-\frac{\cA u}{\bar\rho(1\!+\!b)}+\frac{\pi'_{0}(\bar\rho(1\!+\!b))}{1\!+\!b}\nabla b+
\frac{\pi_1(\bar\rho(1\!+\!b))}{\bar\rho(1\!+\!b)}\nabla{\mathfrak T}+\frac{\pi'_1(\bar\rho(1\!+\!b))}{1\!+\!b}{\mathfrak T}\nabla b=0,\\
&\d_t{\mathfrak T}+u\cdot\nabla{\mathfrak T}+(\bar\cT\!+\!{\mathfrak T})\frac{\pi_1(\bar\rho(1\!+\!b))}{\bar\rho C_v(1\!+\!b)}\div u-\frac\kappa{\bar\rho C_v(1\!+\!b)}
\Delta{\mathfrak T}=\frac{2\mu D(u)\!:\!D(u)+\lambda(\div u)^2}{\bar\rho C_v(1\!+\!b)}\cdotp
\end{aligned}\right.
$$
Then,  setting $\nu\triangleq\lambda+2\mu,$
$\bar\nu\triangleq\nu/\bar\rho,$
 $\chi_0\triangleq\partial_{\rho}P(\bar{\rho},\bar{\mathcal{T}})^{-\frac{1}{2}},$
 and performing  the change of unknowns
$$a(t,x)=b(\bar\nu\chi_0^2t,\bar\nu\chi_0x),\quad \upsilon(t,x)=\chi_0 u(\bar\nu\chi_0^2t,\bar\nu\chi_0x),
\quad
\theta(t,x)=\chi_0\sqrt{\frac{C_v}{\bar\cT}}\:{\mathfrak T}(\bar\nu\chi_0^2t,\bar\nu\chi_0x),$$
we finally obtain
\begin{equation}\label{R-E9}
\left\{\begin{aligned}
&\partial_{t} a+\div \upsilon=f,\\
&\partial_{t} \upsilon-\widetilde\cA\upsilon+\nabla a+\gamma\nabla \theta=g,\\
&\partial_{t} \theta-\beta\Delta \theta+\gamma \div \upsilon=k,
\end{aligned}\right.\end{equation}
with
$$\widetilde\cA\triangleq\frac\cA\nu,
\quad \beta\triangleq\frac\kappa{\nu C_v},\quad
\gamma=\frac{\chi_0}{\bar\rho}\sqrt{\frac{\bar\cT}{ C_v}}\pi_1(\bar\rho),$$
 and where the nonlinear terms $f,$ $g$ and $k$  are given by
$$\displaylines{f\triangleq-\div(a\upsilon),\quad g\triangleq-\upsilon\cdot\nabla\upsilon-I(a)\wt\cA\upsilon-K_1(a)\nabla a-K_2(a)\nabla \theta-\theta\nabla K_3(a)\cr
\hbox{and }\ k\triangleq-\upsilon\cdot\nabla\theta-\beta I(a)\Delta\theta+\frac{Q(\nabla \upsilon, \nabla\upsilon)}{1\!+\!a}-(\tilde{K}_{1}(a)+\tilde{K}_{2}(a)\theta)\div\upsilon}$$
with
$$I(a)\triangleq\frac{a}{1+a}, \quad  K_1(a)\triangleq\frac{\partial_{\rho}P(\bar{\rho}(1\!+\!a),\bar{\mathcal{T}})}{(1\!+\!a)\partial_{\rho}P(\bar{\rho},\bar{\mathcal{T}})}-1, \quad
K_2(a)\triangleq \frac{\chi_0}{\bar\rho}\sqrt{\frac{\bar\cT}{C_v}} \biggl(\frac{\pi_1(\bar{\rho}(1\!+\!a))}{1\!+\!a}-\pi_1(\bar\rho)\biggr),  $$
$$K_3(a)\triangleq \chi_0\sqrt{\frac{\bar\cT}{C_v}}
\int^{a}_{0}\frac{\pi'_{1}(\bar{\rho}(1+z))}{1+z}dz, \quad  \tilde{K}_2(a)\triangleq \frac{\pi_1(\bar{\rho}(1+a))}{C_v\bar{\rho}(1+a)}, $$
$$\tilde{K}_1(a)\triangleq\frac{\chi_0}{\bar\rho}\sqrt{\frac{\bar\cT}{ C_v}}\biggl(\frac{\pi_1(\bar\rho(1+a))}{1+a}-\pi_1(\bar\rho)\biggr),$$
$$Q(A,B)\triangleq\frac{1}{\nu\chi_0} \sqrt{\frac1{\bar\cT C_v}}\bigl(2\mu A:B+\lambda \Tr A \ \Tr B\bigr)\cdotp$$
In fact, the exact value of $K_1, K_2, K_3, \tilde{K}_1$ and $\tilde{K}_2$ will not matter in
our analysis. We shall just use that those functions are smooth and that $K_1(0)=K_2(0)=K_3(0)=\tilde{K}_1(0)=0$.
\medbreak
One can now state the main result of the paper.
\begin{thm}\label{thm2.1} Let the assumptions of Theorem \ref{thm1.1} be in force and $(a,\upsilon,\theta)$ be  the
corresponding global solution.  Let  the real number $s_1$ satisfy
\begin{equation}\label{eq:s1}
\max\Bigl(0,2-\frac d2\Bigr)\leq s_1\leq s_0
\end{equation}
with $s_0\triangleq\frac{2d}p-\frac d2$.
There exists a constant $c>0$ depending only on
$p,d,\lambda,\mu,P, \kappa, C_v,\bar{\rho}, \bar{\mathcal{T}}$ such that if
\begin{equation}\label{R-E10}
{\cD}_{p,0}\triangleq \|(a_0, \upsilon_0, \theta_0)\|_{\dot{B}^{-s_{1}}_{2,\infty}}^\ell \leq c,
\end{equation}
then  for all small enough $\varepsilon>0,$  we have
\begin{equation}\label{R-E11}
{\cD}_{p}(t)\lesssim \bigl({\cD}_{p,0}+\|(\nabla a_0,\upsilon_0)\|^{h}_{\dot B^{\frac dp-1}_{p,1}}+\|\theta_0\|_{\dot B^{\frac dp-2}_{p,1}}^h\bigr)\quad\hbox{for all }\ t\geq0,
\end{equation}
where, setting   $\alpha\triangleq  s_1+\frac d2+\frac 12-\varepsilon,$ the norm ${\cD}_{p}(t)$ is defined by
\begin{multline}\label{R-E12}
{\cD}_{p}(t)\triangleq\sup_{s\in[\varepsilon-s_{1},\frac d2+1]}\|\langle\tau\rangle^{\frac {s_1+s}2}(a,\upsilon,\theta)\|_{L^\infty_t(\dot B^s_{2,1})}^\ell\\
+\|\langle\tau\rangle^\alpha a\|_{\wt L^\infty_t(\dot B^{\frac dp}_{p,1})}^h
+\|\langle\tau\rangle^\alpha \upsilon\|_{\wt L^\infty_t(\dot B^{\frac dp-1}_{p,1})}^h
+\|\langle\tau\rangle^\alpha \theta\|_{\wt L^\infty_t(\dot B^{\frac dp-2}_{p,1})}^h
+\|\tau^\alpha(\nabla \upsilon,\theta)\|_{\wt L^\infty_t(\dot B^{\frac dp}_{p,1})}^h.
\end{multline}
\end{thm}
The above statement deserves some comments:
\begin{itemize}
\item A similar result
may be proved if the physical coefficients $\lambda,$ $\mu$ and $\kappa$ depend smoothly on
the density. Here, we took them constant to avoid more technicalities.
\item To the best of our knowledge, whether well-posedness holds true in critical spaces for the
full Navier-Stokes equations in $\R^2$ is an open question.  At the same time,
it has been proved for slightly more regular data (see \cite{D2,D3}) and we believe
our approach to yield decay estimates for the corresponding solutions (computations
are expected to be wilder, though).
 \item
 For $p=2$ and $s_1=\frac d2,$  hypothesis \eqref{R-E10}  is less restrictive
than the standard $L^1$ condition
first because it only concerns the low frequencies and, second,
because we have  $L^{r}\hookrightarrow\dot{B}^{-s_1}_{2,\infty}$ with $\frac1r=\frac12+\frac{s_1}d\cdotp$
A similar  assumption (for all frequencies and for $p=2$ and $s_1=\frac d2$)  appears in
several recent results :  for the Boltzmann equation in the work by Sohinger and Strain \cite{SS},
 and for  hyperbolic systems with dissipation in the joint paper  of Kawashima with
 the second author \cite{XK1,XK2}.
One can also  mention the work by Guo and Wang \cite{GW} that
replaces the $L^1$ assumption by  homogeneous Sobolev norms  of negative order.
\item
Compared to the barotropic case studied in \cite{DX},  the decay functional ${\cD}_{p}$
contains an additional  decay information on the temperature.
Furthermore, we improved the decay exponent for high frequencies :
in the case $s_1=s_0\triangleq\frac{2d}p-\frac d2$ (which is the only one that has been considered
in \cite{DX}), we get   $\alpha=\frac {2d}{p}+\frac 12-\varepsilon$ instead of~$1.$
We believe exponent $\frac{s_1}2+\frac d2+\frac12$ as well as the upper
bound $s\leq\frac d2+1$ for the first term of $\cD_p$ to be optimal
(see the very end of the present paper for more explanations).
\item If  we replace \eqref{R-E10} with the following slightly stronger hypothesis:
$$
\|(a_0, \upsilon_0, \theta_0)\|_{\dot B^{-s_1}_{2,1}}^\ell \ll 1
$$
then one can take $\ep=0$ in the definitions of $\cD_p$ and $\alpha.$
\item We expect to have similar  decay estimates  for $s_1$ belonging
to the whole range $(1-\frac d2,\frac{2d}p-\frac d2].$
As the analysis for  $s_1<\max(0,2-\frac d2)$  is much more technical
(it relies on tricky product estimates in Besov spaces),
we   chose to present only  the case \eqref{eq:s1}
which is already more general than that  in \cite{DX}.
\end{itemize}
\medbreak
{}From Theorem \ref{thm2.1} and standard embedding,  we readily get  the following algebraic decay rates for the $L^p$ norms of the solution:
\begin{cor}\label{cor2.1}
 Denote $\Lambda^{s}f\triangleq\mathcal{F}^{-1}(|\cdot|^{s}\mathcal{F}f)$.
The solution   of  Theorem \ref{thm1.1} satisfies
$$\displaylines{
\|\Lambda^{s}(\rho-\bar{\rho})\|_{L^p}
\lesssim \bigl(\cD_{p,0}
+\|(\nabla a_0, \upsilon_0)\|_{\dot B^{\frac dp-1}_{p,1}}^h+\|\theta_0\|_{\dot B^{\frac dp-2}_{p,1}}^h\bigr)\langle t\rangle^{-\frac {s_1+s}2}
\  \hbox{ if } \  -s_1<s\leq\frac dp,\cr
\|\Lambda^{s}u\|_{L^p}
\lesssim \bigl(\cD_{p,0}+\|(\nabla a_0, \upsilon_0)\|_{\dot B^{\frac dp-1}_{p,1}}^h+\|\theta_0\|_{\dot B^{\frac dp-2}_{p,1}}^h\bigr)\langle t\rangle^{-\frac {s_{1}+s}2}
\ \hbox{ if } \  -s_1<s\leq\frac dp-1,\cr
\|\Lambda^{s}(\mathcal{T}-\bar{\mathcal{T}})\|_{L^p}
\lesssim \bigl(\cD_{p,0}+\|(\nabla a_0, \upsilon_0)\|_{\dot B^{\frac dp-1}_{p,1}}^h+\|\theta_0\|_{\dot B^{\frac dp-2}_{p,1}}^h\bigr)\langle t\rangle^{-\frac {s_{1}+s}2}
\hbox{ if }  -s_1<s\leq\frac dp-2.}
$$\end{cor}
\begin{rem}
Taking $p=2, $  $s_1=d/2$ and $s=0$ in Corollary \ref{cor2.1} leads back to the standard optimal $L^{1}$-$L^{2}$ decay rates in \eqref{R-E7}, but for much less regular
global solutions.  Also, note
that the derivative index can take both negative and nonnegative values rather than nonnegative integers only, and our results can  thus be regarded as the natural extension of the  classical results of \cite{MN2}.
\end{rem}
Let us briefly present the strategy for proving Theorem \ref{thm2.1}.
The usual approach to get decay estimates  of the type \eqref{R-E7} is to
take advantage of $L^1$-$L^2$ decay estimates for the linear system corresponding
to the left-hand side  of \eqref{R-E9}, treating the nonlinear right-hand side $(f,g,k)$
by means of Duhamel formula (see \cite{Ka,MN2} and references therein). That  basic argument
fails  in the  critical regularity spaces, though, as  one cannot afford any \textit{loss of regularity}
for the high frequency part of the solution (for example, $u\cdot\nabla a$ induces  a loss of one
derivative, while there is no smoothing for $a,$  solution of a  transport equation).
Furthermore, the standard approach completely ignores the fact that the semi-group associated to \eqref{R-E9} behaves
differently for low and high frequencies.
As regards the low frequency part of the solution,  the linearized equations
behave like the heat equation (at least in a $L^2$ type framework)
and  it is possible to adapt the standard method relying on $L^1$-$L^2$
estimate. The only  difference is that  owing to our $L^p$ type assumption
on the high frequencies and our more general low frequency assumption \eqref{R-E10},
  the quadratic  terms are in $L^{r}$ with $\frac1r=\frac12+\frac{s_1}d,$
and it is thus natural to resort to $L^{r}$-$L^2$ estimates (or sometimes
to $\dot B^{-s_1}_{2,\infty}$-$L^2$ ones) rather than  $L^1$-$L^2$ estimates.

We proceed differently for the analysis of the high frequencies decay of the solution.
The idea is to work with  a so-called  ``effective velocity" $w$ (introduced  by D. Hoff in \cite{Hoff}
and first used in the context of critical regularity by B. Haspot in \cite{H2})
such that, up to low order terms,
 the divergence-free part of $u$, the temperature $\theta$ and $w$ fulfill a
 parabolic system while   $a$ satisfies a damped transport equation.
Performing $L^p$ estimates  directly on that system after localization in the Fourier space
and using suitable commutator estimates
to handle the convection term in the density equation, it is possible
to eventually get
 \textit{optimal decay exponents} for high frequencies.
\medbreak
The rest of the paper unfolds as follows. In Section~\ref{sec:3}, we introduce some notation,
recall basic results concerning  Besov spaces, paradifferential calculus, product
and commutator estimates. Section~\ref{sec:4} is devoted to the
proof of Theorem \ref{thm2.1} and of Corollary \ref{cor2.1}.


\setcounter{equation}{0}
\section{Notations, functional spaces and basic tools} \label{sec:3}
Throughout the paper, $C$ stands for a harmless  positive ``constant", the meaning of which is clear
 from the context, and we sometimes write $f\lesssim g$ instead of   $f\leq Cg.$
 The notation  $f\thickapprox g$ means that $f\lesssim g$ and $g\lesssim f$.
   For any  Banach space $X$ and  $f,g\in X,$ we agree that
 $\|(f,g)\|_{X}\triangleq\|f\|_{X}+\|g\|_{X}.$
Finally, for all  $T>0$  and $\rho\in[1,+\infty],$ we denote by $L^\rho_T(X)\triangleq L^{\infty}([0,T];X)$ the set of measurable functions $f:[0,T]\to X$ such that  $t\mapsto \|f(t)\|_{X}$
is in $L^\rho(0,T).$
\medbreak
Let us next recall the definition and a few basic properties
 of Besov spaces  (more details may be found  in e.g. Chap. 2 and 3 of \cite{BCD}).
We start with  a dyadic decomposition
in  Fourier variables: fix some  smooth radial non increasing function $\chi$
supported in $B(0,\frac 43)$ and with value $1$ on $B(0,\frac34),$ then set
$\varphi(\xi)=\chi(\xi/2)-\chi(\xi)$ so that
$$
\qquad\sum_{j\in\Z}\varphi(2^{-j}\cdot)=1\ \hbox{ in }\ \R^d\setminus\{0\}
\quad\hbox{and}\quad \Supp\varphi\subset \big\{\xi\in\R^d : 3/4\leq|\xi|\leq8/3\big\}\cdotp
$$
The homogeneous dyadic blocks $\ddj$ are defined by
$$\ddj f\triangleq\varphi(2^{-j}D)f=\cF^{-1}(\varphi(2^{-j}\cdot)\cF f)=2^{jd}h(2^j\cdot)\star f
\quad\hbox{with}\quad h\triangleq\cF^{-1}\varphi.
$$
The \emph{Littlewood-Paley decomposition} of a general  tempered distribution $f$ reads
\begin{equation}\label{R-E13}
f=\sum_{j\in\Z}\ddj f.
\end{equation}
That equality holds true in the tempered distribution meaning, if  $f$ satisfies
\begin{equation}\label{R-E14}
\lim_{j\rightarrow-\infty}\|\dot S_jf\|_{L^\infty}=0,
\end{equation}
where $\dot S_j$ stands for the low frequency cut-off  $\dot S_j\triangleq\chi(2^{-j}D)$.
\medbreak
In many parts of the paper, we use the notation $z^\ell$ and $z^h$  with
\begin{equation}\label{eq:lh}
z^\ell:=\dot S_{j_0}z\quad\hbox{and}\quad
z^h:=({\rm Id}-\dot S_{j_0})z,
\end{equation}
where the value of the parameter $j_0$ depends  on the coefficients of \eqref{R-E9} through  the proof of the main theorem.
\medbreak
Next, we come to the definition of the homogeneous Besov spaces.
\begin{defn}\label{defn3.1}
 For $s\in\R$ and
$1\leq p,r\leq\infty,$  the homogeneous Besov space $\dot B^s_{p,r}$ is  the
set of tempered distributions $f$ satisfying \eqref{R-E14} and
$$\|f\|_{\dot{B}^{s}_{p,r}}
\triangleq\Big\|\bigl(2^{js}\|\dot{\Delta}_{j}f\|_{L^{p}}\bigr)\Big\|_{\ell^r(\Z)}<\infty.
$$
\end{defn}
In order to state optimal regularity estimates for the heat equation, we need
 the  following  semi-norms  first introduced by J.-Y. Chemin in \cite{Chemin}
for all $0\leq T\leq+\infty,$ $s\in\mathbb{R}$ and  $1\leq r,p,\varrho\leq\infty$:
\begin{equation}\label{eq:tilde}
\|f\|_{\widetilde{L}^{\varrho}_{T}(\dot{B}^{s}_{p,r})}\triangleq\Big\|\bigl(2^{js}\|\dot{\Delta}_{j}f\|_{L^{\varrho}_{T}(L^{p})}\bigr)\Big\|_{\ell^r(\Z)}.\end{equation}
 Index $T$ will be sometimes omitted if $T=+\infty,$ and we  denote
\begin{equation}\label{R-E915}
\widetilde{\mathcal{C}}_{b}(\R_+;\dot{B}^{s}_{p,r})\triangleq\bigl\{f\in \mathcal{C}(\R_+;\dot{B}^{s}_{p,r})\ s.t.\
\|f\|_{\widetilde{L}^{\infty}(\dot{B}^{s}_{p,r})}<\infty\bigr\}\cdotp
\end{equation}
Recall (see e.g. \cite{BCD}) the following optimal regularity estimates for the heat equation:
\begin{prop}\label{prop3.6}
Let $\sigma\in \mathbb{R}$, $(p,r)\in [1,\infty]^2$ and $1\leq \rho_{2}\leq\rho_{1}\leq\infty$.  Let $u$  satisfy
$$\left\{\begin{array}{lll}\d_tu-\mu\Delta u=f,\\
u_{|t=0}=u_0.\end{array}
\right.$$
Then  for all $T>0$ the following a priori estimate is fulfilled:
\begin{equation}\label{R-E28}\mu^{\frac1{\rho_1}}\|u\|_{\tilde L_{T}^{\rho_1}(\dot B^{\sigma+\frac2{\rho_1}}_{p,r})}\lesssim
\|u_0\|_{\dot B^\sigma_{p,r}}+\mu^{\frac{1}{\rho_2}-1}\|f\|_{\tilde L^{\rho_2}_{T}(\dot B^{\sigma-2+\frac2{\rho_2}}_{p,r})}.
\end{equation}
The same estimate holds true  (with a different dependency with respect to  the viscosity coefficients) for the solutions to the following  \emph{Lam\'e system}
\begin{equation*}
\left\{\begin{array}{lll}\d_tu-\mu\Delta u-(\lambda+\mu)\nabla\div u=f,\\
u_{|t=0}=u_0,\end{array}
\right.\end{equation*}
whenever  $\mu>0$ and $\lambda+2\mu>0.$
\end{prop}
 The properties of continuity for the product, commutators and composition involving
 standard Besov norms  remain true  when using  the  semi-norms defined in \eqref{eq:tilde}. The general principle
is  that the time Lebesgue exponent has to be treated according to H\"{o}lder inequality.
Furthermore, Minkowski's inequality allows to compare
$\|\cdot\|_{\widetilde{L}^{\varrho}_{T}(\dot{B}^{s}_{p,r})}$ with the more standard
Lebesgue-Besov semi-norms of  $L^{\varrho}_{T}(\dot{B}^{s}_{p,r})$ as follows:
\begin{equation}\label{R-E15}
\|f\|_{\widetilde{L}^{\varrho}_{T}(\dot{B}^{s}_{p,r})}\leq\|f\|_{L^{\varrho}_{T}(\dot{B}^{s}_{p,r})}\,\,\,
\mbox{if }\,\, r\geq\varrho,\ \ \ \
\|f\|_{\widetilde{L}^{\varrho}_{T}(\dot{B}^{s}_{p,r})}\geq\|f\|_{L^{\varrho}_{T}(\dot{B}^{s}_{p,r})}\,\,\,
\mbox{if }\,\, r\leq\varrho.
\end{equation}

Restricting Besov  norms to the low or high frequencies parts of distributions plays
a fundamental role in our approach.  For that reason, we
shall often use the following notation for some suitable  integer $j_0$\footnote{For technical reasons, it is convenient to have  a small
overlap between low and high frequencies.}:
$$\displaylines{
\|f\|_{\dot B^s_{p,1}}^\ell\triangleq\sum_{j\leq j_0} 2^{js}\|\ddj f\|_{L^p}\quad\hbox{and}\quad
\|f\|_{\dot B^s_{p,1}}^h\triangleq\sum_{j\geq j_0-1} 2^{js}\|\ddj f\|_{L^p},\cr
\|f\|_{\wt L^\infty_T(\dot B^s_{p,1})}^\ell\triangleq\sum_{j\leq j_0} 2^{js}\|\ddj f\|_{L^\infty_T(L^p)}\quad\hbox{and}\quad
\|f\|_{\wt L^\infty_T(\dot B^s_{p,1})}^h\triangleq\sum_{j\geq j_0-1} 2^{js}\|\ddj f\|_{L^\infty_T(L^p)}.}
$$
We shall use
 the following  nonlinear generalization  of the Bernstein inequality  (see e.g. Lemma 8 in \cite{D1}) that
 will be the key to controlling the  $L^p$ norms of  the solution to the spectrally localized
 system \eqref{R-E9}:
\begin{prop}\label{prop3.4}
There exists $c$ depending only on $d,$ $R_1$  and $R_2$
so that  if
\begin{equation}\label{eq:revBerns}
\mathrm{Supp}\,\mathcal{F}f\subset \{\xi\in \mathbb{R}^{d}:
R_{1}\lambda\leq|\xi|\leq R_{2}\lambda\},\end{equation}
 is fulfilled then for all $1<p<\infty,$
\begin{eqnarray}\label{R-E26}&&c\lambda^2\biggl(\frac{p-1}p\biggr)\int_{\R^d}|f|^p\,dx\leq (p-1)\int_{\R^d}|\nabla f|^2|f|^{p-2}\,dx
=-\int_{\R^d}\Delta f\: |f|^{p-2}f\,dx.
\end{eqnarray}
\end{prop}

 Below are embedding properties which are  used several times:
\begin{lem}\label{lem3.2}
\begin{itemize}
  \item For any $p\in[1,\infty]$ we have the  continuous embedding
$$\dot B^0_{p,1}\hookrightarrow L^p\hookrightarrow \dot B^0_{p,\infty}.$$
\item If  $s\in\R,$ $1\leq p_1\leq p_2\leq\infty$ and $1\leq r_1\leq r_2\leq\infty,$
  then $\dot B^{s}_{p_1,r_1}\hookrightarrow
  \dot B^{s-d(\frac1{p_1}-\frac1{p_2})}_{p_2,r_2}.$
  \item The space  $\dot B^{\frac dp}_{p,1}$ is continuously embedded in   the set  of
bounded  continuous functions (going to $0$ at infinity if    $p<\infty$).
  \end{itemize}
\end{lem}
\medbreak
Let us also mention the following  interpolation inequality
\begin{lem}\label{lem3.3}
Let $1\leq p,r_1,r_2,r\leq\infty,$ $\sigma_1\not=\sigma_2$ and $\theta\in(0,1)$. Then
\begin{equation}\label{R-E19}
  \|f\|_{\dot B^{\theta\sigma_2+(1-\theta)\sigma_1}_{p,r}}\lesssim\|f\|_{\dot B^{\sigma_1}_{p,r_1}}^{1-\theta}
  \|f\|_{\dot B^{\sigma_2}_{p,r_2}}^\theta.
\end{equation}
\end{lem}

As already pointed out, the low frequency hypothesis is less restrictive than the usual $L^q.$
This is a consequence of the  following embedding:
\begin{lem}\label{lem3.4}
Suppose that $\sigma>0$ and $1\leq q<2$. One has
\begin{equation}\label{R-E151}
\|f\|_{\dot{B}^{-\sigma}_{r,\infty}}\lesssim \|f\|_{L^{q}}
\quad\hbox{with }\ \frac1q-\frac1r=\frac\sigma d\cdotp
\end{equation}
\end{lem}

The following product laws and commutator estimates proved in
 e.g. \cite{BCD} and \cite{DX}
will play a fundamental  role in the analysis of  the bilinear
terms of \eqref{R-E3}.
\begin{prop}\label{prop3.1}
Let $\sigma>0$ and $1\leq p,r\leq\infty$. Then $\dot{B}^{\sigma}_{p,r}\cap L^{\infty}$ is an algebra and
\begin{equation}\label{R-E20}
\|fg\|_{\dot{B}^{\sigma}_{p,r}}\lesssim \|f\|_{L^{\infty}}\|g\|_{\dot{B}^{\sigma}_{p,r}}+\|g\|_{L^{\infty}}\|f\|_{\dot{B}^{\sigma}_{p,r}}.
\end{equation}
Let the real numbers $\sigma_{1},$ $\sigma_{2},$ $p_1$  and $p_2$ be such that
$$
\sigma_1+\sigma_2>0,\quad \sigma_1\leq\frac d{p_1},\quad\sigma_2\leq\frac d{p_2},\quad
\sigma_1\geq\sigma_2,\quad\frac1{p_1}+\frac1{p_2}\leq1.
$$
Then we have
\begin{equation}\label{R-E21}\|fg\|_{\dot{B}^{\sigma_{2}}_{q,1}}\lesssim \|f\|_{\dot{B}^{\sigma_{1}}_{p_1,1}}\|g\|_{\dot{B}^{\sigma_{2}}_{p_2,1}}\quad\hbox{with}\quad
\frac1{q}=\frac1{p_1}+\frac1{p_2}-\frac{\sigma_1}d\cdotp\end{equation}
Finally, for exponents $\sigma>0$ and $1\leq p_1,p_2,q\leq\infty$ satisfying
$$
\frac{d}{p_1}+\frac{d}{p_2}-d\leq \sigma \leq\min\biggl(\frac d{p_1},\frac d{p_2}\biggr)\quad\hbox{and}\quad \frac 1q=\frac 1{p_1}+\frac 1{p_2}-\frac\sigma d,
$$
we have
\begin{equation}\label{R-E22}\|fg\|_{\dot{B}^{-\sigma}_{q,\infty}}\lesssim
\|f\|_{\dot{B}^{\sigma}_{p_1,1}}\|g\|_{\dot{B}^{-\sigma}_{p_2,\infty}}.
\end{equation}
\end{prop}

Proposition \ref{prop3.1} is not enough to handle the case $2p>d$ in the proof of Theorem \ref{thm2.1}.
We shall make use of the following estimates proved recently in \cite{DX}.
\begin{prop}\label{prop3.2}  Let $j_0\in\Z,$ and denote $z^\ell\triangleq\dot S_{j_0}z,$  $z^h\triangleq z-z^\ell$ and, for any $s\in\R,$
$$
\|z\|_{\dot B^s_{2,\infty}}^\ell\triangleq\sup_{j\leq j_0}2^{js} \|\ddj z\|_{L^2}.
$$
There exists a universal integer $N_0$ such that  for any $2\leq p\leq 4$ and $\sigma>0,$ we have
\begin{eqnarray}\label{R-E23}
&&\|f g^h\|_{\dot B^{-s_0}_{2,\infty}}^\ell\leq C \bigl(\|f\|_{\dot B^\sigma_{p,1}}+\|\dot S_{j_0+N_0}f\|_{L^{p^*}}\bigr)\|g^h\|_{\dot B^{-\sigma}_{p,\infty}}\\\label{R-E24}
&&\|f^h g\|_{\dot B^{-s_0}_{2,\infty}}^\ell
\leq C \bigl(\|f^h\|_{\dot B^\sigma_{p,1}}+\|\dot S_{j_0+N_0}f^h\|_{L^{p^*}}\bigr)\|g\|_{\dot B^{-\sigma}_{p,\infty}}
\end{eqnarray}
with  $s_0\triangleq \frac{2d}p-\frac d2$ and $\frac1{p^*}\triangleq\frac12-\frac1p,$
and $C$ depending only on $j_0,$ $d$ and $\sigma.$
\end{prop}
The terms in System \eqref{R-E9}  corresponding to the functions
$K_1, K_2, K_3, \tilde{K}_1$ and $\tilde{K}_2$ will be
  bounded thanks to the following classical result:
\begin{prop}\label{prop:compo}
Let $F:\R\rightarrow\R$ be  smooth
with $F(0)=0.$
For  all  $1\leq p,r\leq\infty$ and  $\sigma>0$ we have
$F(f)\in \dot B^\sigma_{p,r}\cap L^\infty$  for  $f\in \dot B^\sigma_{p,r}\cap L^\infty,$  and
\begin{eqnarray}\label{R-E25}
\|F(f)\|_{\dot B^\sigma_{p,r}}\leq C\|f\|_{\dot B^\sigma_{p,r}}
\end{eqnarray}
with $C$ depending only on $\|f\|_{L^\infty},$ $F'$ (and higher derivatives),  $\sigma,$ $p$ and $d.$
\end{prop}
\begin{rem}  As  Theorem \ref{thm1.1} involves $\tilde L_T^\varrho(\dot B^s_{p,r})$ semi-norms,
we shall very often use Propositions \ref{prop3.1}, \ref{prop3.2} and \ref{prop:compo} adapted to those spaces.
The general rule is that exactly the same estimates hold true, once the time Lebesgue exponents
have been treated according to H\"older inequality. \end{rem}


\section{The proof of decay estimates} \setcounter{equation}{0}\label{sec:4}

In this section, we prove Theorem \ref{thm2.1}.
We start with the global solution $(a,\upsilon,\theta)$ given by
 Theorem \ref{thm1.1}, that satisfies \eqref{R-E5} and thus in particular
 \begin{equation}\label{eq:smalla}
 \|a\|_{\wt L^\infty(\dot B^{\frac dp}_{p,1})}\ll1.
 \end{equation}
 Then there only remains to prove  \eqref{R-E11}.
 The overall strategy is to combine  frequency-localization of the linear part of the full system
 with the Duhamel principle to handle the nonlinear terms.
   \medbreak
    We shall proceed in three steps.
      Step 1 is dedicated to the proof of decay estimates for the low  frequency part of
 $(a,\upsilon,\theta),$  that is
 $$\cD_{p,1}(t)\triangleq  \sup_{s\in[\varepsilon-s_{1},\frac d2+1]}\|\langle\tau\rangle^{\frac {s_1+s}2}(a,\upsilon,\theta)\|_{L^\infty_t(\dot B^s_{2,1})}^\ell.$$
 To this end,  we shall  exhibit the low frequency decay exponents
for  the linear system corresponding
 to the left-hand side of  \eqref{R-E9}. This  will  be achieved by means of  a simple energy method
 on the spectrally localized system, \emph{without  computing
 the Green function of the linear system}.
 Denoting by $(E(t))_{t\geq0}$ the corresponding semi-group, we shall get
\begin{equation}\label{R-E49}
\sup_{t\geq0}\,\langle t\rangle^{\frac {s_1+s}2}\|E(t)U_{0}\|_{\dot B^s_{2,1}}^\ell\lesssim \|U_{0}\|_{\dot{B}^{-s_{1}}_{2,\infty}}^\ell\quad\hbox{if }\ s>-s_1.
\end{equation}
Then taking advantage of Duhamel formula reduces the proof to that of suitable
decay estimates in the space $\dot B^{-s_1}_{2,\infty}$ for    all the nonlinear terms in $f,$ $g$ and $h.$
  Using the embedding $L^r\hookrightarrow \dot B^{-s_1}_{2,\infty}$  with $\frac1r=\frac{s_1}d+\frac12$
  and   the obvious product law $L^{2r}\times L^{2r}\to L^{r},$ it turns out to be
  (almost) enough to exhibit appropriate bounds for the solution in $L^{2r}.$
  In fact, that strategy  fails only  if $p>\frac d2,$ owing to the term $I(a)\Delta\theta^h$
 which, according to Theorem \ref{thm1.1}, is only  in $\dot B^{\frac dp-2}_{p,1}$ for a.a. $t\geq0.$
 Clearly, if  $\frac dp-2<0$ then  one cannot expect that term to be in any Lebesgue space.
 To overcome that difficulty, we will have to use the more elaborate product
 laws in Besov spaces stated in Proposition \ref{prop3.2}.
\smallbreak
 Step 2 is devoted to bounding\footnote{Recall that  $\alpha\triangleq  s_1+\frac d2+\frac12-\varepsilon.$}
 $$\cD_{p,2}(t)\triangleq\|\langle\tau\rangle^{\alpha}a\|_{\wt L^\infty_t(\dot B^{\frac dp}_{p,1})}^h+\|\langle\tau\rangle^{\alpha}\upsilon\|_{\wt L^\infty_t(\dot B^{\frac dp-1}_{p,1})}^h+\|\langle\tau\rangle^{\alpha}\theta\|_{\wt L^\infty_t(\dot B^{\frac dp-2}_{p,1})}^h.
 $$
 To this end, following  Haspot's approach in   \cite{H2},
 we introduce the \emph{effective velocity}
  \begin{equation}\label{eq:w} w\triangleq\nabla(-\Delta)^{-1}(a-\mathrm{div}\,\upsilon).\end{equation}
 Then, up to low order terms,  $w,$ $\theta$   and the divergence free part  of the velocity
 satisfy a heat equation, while  $a$ fulfills  a \emph{damped} transport equation.
  That way of rewriting the full system allows to  avoid the loss of one derivative arising from the convection term
in the density equation. Basically, this is  the same strategy as for   the barotropic Navier-Stokes equations (see \cite{DX}).
In the polytropic case however,  one has to  perform  low and high frequency decompositions
of the  (new) nonlinear terms involving $\theta,$ for  $\theta^h$  is  less regular than  $\upsilon^h$ by one derivative.
\smallbreak
 In the last step,  we establish a gain of regularity and decay altogether for the high frequencies of $\upsilon$ and $\theta$, namely we bound
 $$
 \cD_{p,3}(t)\triangleq \|\tau^\alpha(\nabla \upsilon,\theta)\|_{\wt L^\infty_t(\dot B^{\frac dp}_{p,1})}^h.
 $$
 This  strongly relies on the
maximal regularity estimates for the heat and Lam\'e semi-groups (see Proposition \ref{prop3.6}).
Compared to our previous paper \cite{DX},   we  found a way to improve substantially the decay rate exhibited in $\cD_{p,3}:$
  we  get  $t^{-\alpha}$ instead of just $t^{-1}$.

\subsubsection*{Step 1: Bounds for the low frequencies}

Let  $(E(t))_{t\geq0}$ be the semi-group associated to the left-hand side of \eqref{R-E9}. The standard Duhamel principle gives
\begin{equation}\label{R-E30}
\left(\begin{array}{c}a(t)\\ \upsilon (t)\\ \theta (t)\end{array}\right)
={E(t)}\left(\begin{array}{c}a_{0}\\ \upsilon_{0}\\ \theta_{0} \end{array}\right)
\!+\!\int_0^t E(t-\tau)\left(\begin{array}{c}f(\tau)\\g(\tau)\\ k(\tau)\end{array}\right)d\tau.
\end{equation}

The following lemma states  that the low frequency part of
 $(a_{L}, \upsilon_{L}, \theta_{L})\triangleq E(t)(a_{0}\,\upsilon_{0}, \theta_{0})$
 behaves essentially  like  the solution to the heat equation.
 \begin{lem}\label{lem4.1}
Let $(a_{L}, \upsilon_{L}, \theta_{L})$ be the solution to the following system
\begin{equation}\label{R-E31}
\left\{\begin{array}{l}
\partial_{t} a_{L}+\div \upsilon_{L}=0,\\
\partial_{t} \upsilon_{L}-\widetilde\cA\upsilon_{L}+\nabla a_{L}+\gamma\nabla \theta_{L}=0,\\
\partial_{t} \theta_{L}-\beta\Delta \theta_{L}+ \gamma\div \upsilon_{L}=0,
\end{array}\right.
\end{equation}
with the initial data
\begin{equation}\label{R-E32}
(a_{L},\upsilon_{L},\theta_{L})|_{t=0}=(a_{0},\upsilon_{0},\theta_{0}).
\end{equation}
Then, for any $j_{0}\in \Z$, there exists a  positive constant $c_{0}=c_0(\lambda,\mu,\beta,\gamma,j_0)$
such that
\begin{eqnarray}\label{R-E33}
\|(a_{L,j},\upsilon_{L,j},\theta_{L,j})(t)\|_{L^2}\lesssim e^{-c_02^{2j}t}\|(a_{0,j},\upsilon_{0,j},\theta_{0,j})\|_{L^2}
\end{eqnarray}
for $t\geq0$ and $j\leq j_{0}$, where we set $z_j=\dot{\Delta}_jz$ for any $z\in \mathcal{S}'(\mathbb{R}^{d})$.
\end{lem}
\begin{proof}
Set $\omega_{L}\triangleq|D|^{-1}\div\upsilon_{L}$ with $\mathcal{F}(|D|^{s}f)(\xi)\triangleq |\xi|^s\hat{f}(\xi)(s\in \R)$.
Remembering that $\cA=\nu^{-1}(\mu\Delta+(\lambda+\mu)\nabla\div)$ and that $\nu=\lambda+2\mu,$ we get
\begin{equation}\label{R-E34}
\left\{\begin{array}{l}
\partial_{t} a_{L}+|D|\omega_{L}=0,\\
\partial_{t} \omega_{L}-\Delta \omega_{L}-|D|a_{L}-\gamma|D|\theta_{L}=0,\\
\partial_{t} \theta_{L}-\beta\Delta \theta_{L}+\gamma|D|\omega_{L}=0.
\end{array}\right.
\end{equation}
 Let  $(A,\Omega,\Theta)$  be the
 Fourier transform of  $(a_{L}, \omega_{L}, \theta_{L})$. Then,  we have
 for all $\varrho\triangleq|\xi|,$
 \begin{equation}\label{R-E34b}
\left\{\begin{array}{l}
\partial_{t} A+\varrho\Omega=0,\\
\partial_{t} \Omega+\rho^2\Omega-\rho A-\gamma\rho\Theta=0,\\
\partial_{t} \Theta+\beta\rho^2\Theta+\gamma\rho\Omega=0,
\end{array}\right.
\end{equation}
 from which we  get the following three identities:
\begin{eqnarray}\label{R-E35}
\frac{1}{2}\frac{d}{dt}|A|^2+\varrho\mathrm{Re}(A\bar{\Omega})=0,\\
\label{R-E36}
\frac{1}{2}\frac{d}{dt}|\Omega|^2+\varrho^2|\Omega|^2-\varrho\mathrm{Re}(A\bar{\Omega})
-\gamma\varrho\mathrm{Re}(\Theta\bar{\Omega})=0,\\\label{R-E37}
\frac{1}{2}\frac{d}{dt}|\Theta|^2+\beta\varrho^2|\Theta|^2+\gamma\varrho\mathrm{Re}(\Theta\bar{\Omega})=0,
\end{eqnarray}
where $\bar{f}$ indicates the complex conjugate of a function $f$.
\medbreak
By adding up \eqref{R-E35}, \eqref{R-E36} and  \eqref{R-E37}, we obtain
\begin{equation}\label{R-E38}
\frac{1}{2}\frac{d}{dt}|(A,\Omega,\Theta)|^2+\wt\beta\varrho^2|(\Omega,\Theta)|^2\leq0\quad\hbox{with }\ \wt\beta\triangleq\min(1,\beta).
\end{equation}
Next, to track  the dissipation of $A$,  we use the fact that
\begin{equation}\label{R-E39}
\frac{d}{dt}[-\mathrm{Re}(A\bar{\Omega})]+\varrho|A|^2-\varrho|\Omega|^2-\varrho^2\mathrm{Re}(A\bar{\Omega})+\gamma\varrho\mathrm{Re}(\Theta\bar{A})=0.
\end{equation}
Combining with \eqref{R-E35},  we deduce that
\begin{equation}\label{R-E40}
\frac{1}{2}\frac{d}{dt}[|\varrho A|^2-2\varrho\mathrm{Re}(A\bar{\Omega})]+\varrho^2|A|^2-\varrho^2|\Omega|^2
+\gamma\varrho^2\mathrm{Re}(\Theta\bar{A})=0.
\end{equation}
Therefore, introducing the ``Lyapunov functional''
$$\mathcal{L}^2_{\varrho}(t)\triangleq |(A,\Omega,\Theta)|^2+K\bigl(|\varrho A|^2-2\varrho\mathrm{Re}(A\bar{\Omega})\bigr)$$
for some $K>0$ (to be chosen hereafter), we get from \eqref{R-E38} and \eqref{R-E40} that
$$
\frac{1}{2}\frac{d}{dt}\mathcal{L}^2_{\varrho}+
\varrho^2(K|A|^2+(\wt\beta-K)|\Omega|^2+\wt\beta|\Theta|^2)
+K\gamma\varrho^2\mathrm{Re}(\Theta\bar{A})\leq0.
$$
Then, taking advantage  of the following  Young inequality
$$
\bigl|\gamma\mathrm{Re}(\Theta\bar{A})\bigr|\leq\frac12|A|^2+\frac{\gamma^2}2|\Theta|^2
$$
then choosing $K$ so that $K\gamma^2=\wt\beta-K,$
we end up with
\begin{equation}\label{R-E42}
\frac{d}{dt}\mathcal{L}^2_{\varrho}+
\wt\beta\varrho^2\biggl(\frac2{1+\gamma^2}|A|^2 +\frac{\gamma^2}{1+\gamma^2}|\Omega^2| +2|\Theta|^2)\biggr)\leq0.
\end{equation}
Using again Young's inequality, we discover that there exists some constant $C_{0}>0$ depending
only on $\beta$ and $\gamma$ so  so that
\begin{equation}\label{R-E43}
C_{0}^{-1}\mathcal{L}^2_{\varrho}\leq |(A,\varrho A,\Omega,\Theta)|^2\leq C_{0}\mathcal{L}^2_{\varrho}.
\end{equation}
This in particular implies that for all fixed $\rho_0>0$ we have for some constant $c_0$
depending only on $\rho_0,$ $\beta$ and $\gamma,$
$$
\wt\beta\biggl(\frac2{1+\gamma^2}|A|^2 +\frac{\gamma^2}{1+\gamma^2}|\Omega^2| +2|\Theta|^2\biggr)\geq c_0\cL_\rho^2\quad\hbox{for all }\  0\leq\varrho\leq\varrho_0.
$$
Therefore, reverting to \eqref{R-E42},
$$
\cL_\rho^2(t)\leq e^{-c_0\varrho^2t}\cL_\rho^2(0),$$
whence using \eqref{R-E43},
\begin{equation}\label{R-E46}
|(A,\Omega,\Theta)|\leq C e^{-\frac{c_{0}}{2}t\varrho^2}|(A,\Omega,\Theta)(0)|
\quad\hbox{for all }\ t\geq0\ \hbox{ and }\ 0\leq\varrho\leq\varrho_0.
\end{equation}
Multiplying both sides by $\varphi(2^{-j}\xi)$ with $|\xi|=\varrho,$  then taking the $L^2$
norm and using
Fourier-Plancherel theorem,  we end up for all $j\leq j_0$ with
\begin{equation}\label{R-E48}
\|(a_{L,j},\omega_{L,j},\theta_{L,j})(t)\|_{L^2}\lesssim e^{-\frac{c_{0}}22^{2j}t}\|(a_{L,j},\omega_{L,j},\theta_{L,j})(0)\|_{L^2}.
\end{equation}
Now, as  the divergence free part $\cP u_L$ of $u_L$ satisfies the heat equation
$$
\d_t\cP u_L-\wt\mu\Delta u_L=0,
$$
we have  for all $j\in\Z,$
\begin{equation}\label{R-E48b}
\|\cP u_{L,j}(t)\|_{L^2}\leq e^{-{\wt\mu}2^{2j}t}\|\cP u_{L,j}(0)\|_{L^2}\quad\hbox{with }\ \wt\mu:=\mu/\nu.
\end{equation}
Then putting \eqref{R-E48} and \eqref{R-E48b} together completes the proof of the lemma.
\end{proof}
\medbreak
Set $U\triangleq(a, \upsilon, \theta)$ and $U_{0}\triangleq(a_{0}, \upsilon_{0}, \theta_{0})$.
{}From  Lemma \ref{lem4.1} and the obvious inequality
$$
\sup_{t\geq0} \sum_{j\in\Z} t^{\frac{s_1+s}2} 2^{j(s_1+s)} e^{-c_02^{2j}t}\leq C_s<+\infty\quad\hbox{if }\ s+s_1>0,
$$
we  get  \eqref{R-E49}  (see \cite{DX} for more  details). Hence we have
\begin{eqnarray}\label{R-E50}
&& \biggl\|\int_0^tE(t-\tau)(f,g,k)(\tau)\,d\tau\biggr\|^\ell_{\dot B^s_{2,1}}
\lesssim\int_0^t\langle t-\tau\rangle^{-\frac{s_1+s}2}\|(f,g,k)(\tau)\|_{\dot B^{-s_1}_{2,\infty}}^\ell\,d\tau.
\end{eqnarray}
We claim that if $p$ and $s_1$ fulfill \eqref{eq:p} and \eqref{eq:s1}, respectively,  then we have for all $t\geq0,$
\begin{equation}\label{R-E51}
\int_0^t\langle t-\tau\rangle^{-\frac {s_1+s}2} \|(f, g, k)(\tau)\|_{\dot B^{-s_1}_{2,\infty}}^\ell\,d\tau\lesssim\langle t\rangle^{-\frac {s_1+s}2}
\bigl(\cD^2_{p}(t)+\cX^2_{p}(t)\bigr).
\end{equation}
In order to prove our claim, we shall use repeatedly   the following
obvious  inequality that is satisfied whenever $0\leq\sigma_1\leq\sigma_2$ and $\sigma_2>1$:
\begin{equation}\label{R-E29}
\int_0^t\langle t-\tau\rangle^{-\sigma_1}\langle \tau\rangle^{-\sigma_2}\,d\tau
\lesssim\langle t\rangle^{-\sigma_1}\ \hbox{ with }\   \langle t\rangle\triangleq \sqrt{1+t^2}.
\end{equation}
We shall also take advantage of the embeddings
$L^r\hookrightarrow \dot B^{-s_1}_{2,\infty}$
and $\dot B^\sigma_{p,1}\hookrightarrow  L^{2r}$
with
$$
s_1\triangleq \frac dr-\frac d2\quad\hbox{and}\quad
\sigma\triangleq \frac dp-\frac d{2r}\cdotp
$$
Let us emphasize that Condition \eqref{eq:s1} is equivalent to $\frac p2 \leq r \leq \min(2,\frac d2)\cdotp$
\medbreak
All the terms in $f,$ $g$ and $h,$ but $I(a)\Delta\theta^h$ will be treated thanks to
the following inequalities that follow from H\"older inequality and
the above embeddings:
\begin{equation}\label{R-E41}
\|FG\|_{\dot B^{-s_1}_{2,\infty}}\lesssim \|FG\|_{L^r}\leq \|F\|_{L^{2r}}\|G\|_{L^{2r}}
\lesssim \|F\|_{\dot B^\sigma_{p,1}}\|G\|_{\dot B^{\sigma}_{p,1}}.
\end{equation}
We shall often use the fact that, because $\sigma+2\leq \frac d2+1,$
\begin{eqnarray}\label{R-E54}
\sup_{0\leq\tau\leq t}\langle \tau\rangle^{\frac{d}{4r}+\frac k2}\|(\nabla^ka^{\ell},\nabla^k\upsilon^\ell,\nabla^k\theta^\ell)(\tau)\|_{L^{2r}}\lesssim \cD_{p}(t)\quad\hbox{for }\ k=0,1,2.
\end{eqnarray}
This is just a consequence of
$$
\|D^kz\|_{L^{2r}}^\ell\lesssim \|z\|_{\dot B^{\sigma+k}_{p,1}}^\ell \lesssim
\|z\|^\ell_{\dot B^{k+\frac d2-\frac d{2r}}_{2,1}}
$$
and of the fact that we have $-s_1<k+\frac d2-\frac d{2r}\leq \frac d2+1$ for
$k\in\{0,1,2\}.$
\medbreak
Then using also  that $\sigma\leq \frac dp-1$ (as $r\leq\frac d2$),
one can write that
$$\|z\|^h_{\dot B^\sigma_{p,1}}\lesssim \|z\|^h_{\dot B^{\frac dp-1}_{p,1}}.$$
Therefore as obviously $\alpha\geq \frac d{4r}+\frac12$ for small enough $\ep,$
 we  have
\begin{equation}\label{R-E45}
\sup_{0\leq\tau\leq t}\langle\tau\rangle^{\frac d{4r}} \|(\upsilon,a)(\tau)\|_{L^{2r}}+
\sup_{0\leq\tau\leq t}\langle\tau\rangle^{\frac d{4r}+\frac12} \|\nabla a(\tau)\|_{L^{2r}}
\lesssim\cD_p(t).
\end{equation}
Combining with \eqref{R-E41}, one can thus  bound  the terms
$\upsilon\cdot\nabla a$ and  $K_1(a)\nabla a$  as follows:
$$
\|\upsilon\cdot\nabla a\|_{\dot B^{-s_1}_{2,\infty}}+\|K_1(a)\nabla a\|_{\dot B^{-s_1}_{2,\infty}}
\lesssim \|(a,\upsilon)\|_{L^{2r}}\|\nabla a\|_{L^{2r}}
\lesssim \langle t\rangle^{-\frac d{2r}-\frac12}\cD_p^2(t).
$$
Now, as $\frac d{2r}+\frac12\geq\frac12(s_1+s)$ for all $s\leq1+\frac d2,$ we get for all $t\geq0,$
$$\sup_{-s_1+\ep\leq s\leq1+\frac d2}
\int_0^t\langle t-\tau\rangle^{-\frac {s_1+s}2}
\bigl(\|\upsilon\cdot\nabla a\|_{\dot B^{-s_1}_{2,\infty}}^\ell
\!+\!\|K_1(a)\nabla a\|_{\dot B^{-s_1}_{2,\infty}}^\ell\bigr)
d\tau\lesssim\langle t\rangle^{-\frac {s_1+s}2}
\bigl(\cD^2_{p}(t)\!+\!\cX^2_{p}(t)\bigr).
$$

The terms $a\,\div\upsilon^\ell,$ $\upsilon\cdot\nabla\upsilon^\ell,$ $I(a)\Delta\upsilon^\ell,$
$K_2(a)\nabla\theta^\ell,$ $\theta^\ell\nabla K_3(a),$ $\upsilon\cdot\nabla\theta^\ell,$
$I(a)\Delta\theta^\ell,$ $\wt K_1(a)\div\upsilon^\ell$ and
$\wt K_2(a)\theta^\ell\div\upsilon^\ell$  may be handled along the same lines.
\medbreak
To deal with  the other terms in $f,$ $g$ and $k,$
we shall also often use the fact that
\begin{equation}\label{R-E44}
\|t^{\frac{s_1}2+\frac d4-\frac\ep2}(a,\upsilon,\theta)\|_{\wt L_T^\infty(\dot B^{\frac dp}_{p,1})}
\lesssim\cD_p(t).
\end{equation}
This  may be  proved by  decomposing functions $a,$ $\upsilon$ and $\theta$
 in low and high frequencies, using the definition of $\cD_p(t)$ and the fact that, because $p\geq2$
and $\alpha\geq \frac{s_1}2+\frac d4-\frac\ep2,$
$$\begin{aligned}
\|t^{\frac{s_1}2+\frac d4-\frac\ep2}(a,\upsilon,\theta)\|_{\wt L_T^\infty(\dot B^{\frac dp}_{p,1})}
&\lesssim \|t^{\frac{s_1}2+\frac d4-\frac\ep2}(a,\upsilon,\theta)\|^\ell_{\wt L_T^\infty(\dot B^{\frac d2}_{2,1})}
+\|t^{\frac{s_1}2+\frac d4-\frac\ep2}(a,\upsilon,\theta)\|^h_{\wt L_T^\infty(\dot B^{\frac dp}_{p,1})}\\
&\lesssim \|t^{\frac{s_1}2+\frac d4-\frac\ep2}(a,\upsilon,\theta)\|^\ell_{L_T^\infty(\dot B^{\frac d2-\ep}_{2,1})}
+\|t^{\alpha}(a,\upsilon,\theta)\|^h_{\wt L_T^\infty(\dot B^{\frac dp}_{p,1})}.
\end{aligned}
$$
As an example, let us  bound   $K_2(a)\nabla \theta^{h}.$
Then we use that, owing to \eqref{R-E41},  if $t\geq2,$
$$\displaylines{
\int_0^t\langle t-\tau\rangle^{-\frac {s_1+s}2}\|(K_2(a)\nabla \theta^{h})(\tau)\|_{\dot B^{-s_1}_{2,\infty}}\,d\tau\hfill\cr\hfill \lesssim
\int_0^t\langle t-\tau\rangle^{-\frac{s_{1}+s}2} \|a(\tau)\|_{L^{2r}}\|\nabla \theta^{h}(\tau)\|_{L^{2r}}\,d\tau =
\Big(\int^{1}_{0}+\int^{t}_{1}\Big)(\cdot\cdot\cdot)d\tau\triangleq I_{1}+I_{2}.}$$
 Using the fact that
$$
\|\nabla\theta^h\|_{L^{2r}}\lesssim \|\theta\|^h_{\dot B^{\sigma+1}_{p,1}}
\lesssim \|\theta\|^h_{\dot B^{\frac dp}_{p,1}},
$$
and remembering the definitions of $\cX_{p}(t)$ and $\cD_{p}(t),$ and Inequality
\eqref{R-E45},  we obtain
\begin{eqnarray}\label{R-E58}
I_{1}&\!\!\lesssim\!\!& \langle t\rangle^{-\frac{s_{1}+s}2} \Big(\sup_{0\leq \tau \leq 1}\|a(\tau)\|_{L^{2r}}\Big)\int^{1}_{0}\|\theta(\tau)\|^{h}_{\dot{B}^{\frac{d}{p}}_{p,1}}\,d\tau
\nonumber\\ &\!\!\lesssim\!\!& \langle t\rangle^{-\frac{s_{1}+s}2} \cD_{p}(1)\cX_{p}(1).
\end{eqnarray}
Next, because  $\langle \tau\rangle \thickapprox \tau$ when $\tau\geq1,$ one has
for all $-s_1<s\leq\frac d2+1,$
$$\begin{aligned}
I_{2}&=\int_1^t\langle t-\tau\rangle^{-\frac {s_{1}+s}2}\langle \tau\rangle^{-\frac {d}{2r}-\frac12}
\bigl(\langle\tau\rangle^{\frac {d}{4r}}\|a(\tau)\|_{L^{2r}}\bigr)\bigl(\tau^{\frac {d}{4r}+\frac12}\|\nabla \theta^h(\tau)\|_{L^{2r}}\bigr)\,d\tau
\\ &\lesssim \sup_{1\leq \tau \leq t}\!\Big(\langle\tau\rangle^{\frac {d}{4r}}\|a(\tau)\|_{L^{2r}}\!\Big)
\sup_{1\leq \tau \leq t}\!\Big(\tau^{\alpha}\|\theta(\tau)\|^{h}_{\dot{B}^{\frac{d}{p}}_{p,1}}\!\Big)\!\!\int_1^t\!\langle t-\tau\rangle^{-\frac{s_{1}+s}2}
\langle \tau\rangle^{-\frac d{2r}-\frac12}\,d\tau
\\ &\lesssim \langle t\rangle^{-\frac{s_{1}+s}2} \cD_{p}^2(t).
\end{aligned}$$
Therefore, for $t\geq2$, we arrive at
\begin{equation}\label{R-E60}
\int_0^t\langle t-\tau\rangle^{-\frac {s_{1}+s}2}\|(K_2(a)\nabla \theta^{h})(\tau)\|_{\dot B^{-s_1}_{2,\infty}}\,d\tau \lesssim\langle t\rangle^{-\frac{s_{1}+s}2} \Big(\cD_{p}(t)\cX_{p}(t)+\cD_{p}^2(t)\Big)\cdotp
\end{equation}
The case $t\leq2$ is obvious as $\langle t\rangle\thickapprox1$ and
$\langle t-\tau\rangle\thickapprox1$ for $0\leq\tau\leq t\leq 2,$ and
\begin{eqnarray}\label{R-E61}
\int_0^t\|K_2(a)\nabla \theta^{h}\|_{L^{r}}\,d\tau\leq \|a\|_{L^\infty_t(L^{2r})}\|\nabla \theta^{h}\|_{L_t^1(L^{2r})}\lesssim \cD_{p}(t)\cX_{p}(t).
\end{eqnarray}
The terms with $a\div\upsilon^h,$ $I(a)\Delta\upsilon^h,$ $\upsilon\cdot\nabla\upsilon^h,$  $\theta^{h}\nabla K_3(a)$, $\upsilon\cdot\nabla\theta^{h}$ and  $\widetilde{K}_1(a)\div \upsilon^{h}$ may be treated along the same lines.
For the term corresponding to  $\frac{Q(\nabla \upsilon, \nabla\upsilon)}{1+a}$ one may write, thanks
to \eqref{eq:smalla},
$$\begin{aligned}
\int_{0}^{t}\langle t-\tau\rangle^{-\frac {s_{1}+s}{2}}\Big\|\frac{Q(\nabla \upsilon, \nabla\upsilon)}{1+a}\Big\|^{\ell}_{\dot{B}^{-s_{1}}_{2,\infty}}\,d\tau&\lesssim \int_{0}^{t}\langle t-\tau\rangle^{-\frac {s_{1}+s}{2}}\|\nabla \upsilon(\tau)\|^2_{L^{2r}}\,d\tau\\
&\lesssim \int_{0}^{t}\langle t-\tau\rangle^{-\frac {s_{1}+s}{2}}
\bigl(\|\nabla \upsilon^\ell(\tau)\|^2_{L^{2r}}+\|\nabla \upsilon^h(\tau)\|^2_{L^{2r}}\bigr)\,d\tau.
\end{aligned}
$$
Now, \eqref{R-E54} implies that
\begin{multline}\label{R-E78}
\int_{0}^{t}\langle t-\tau\rangle^{-\frac {s_{1}+s}{2}}\|\nabla \upsilon^{\ell}(\tau)\|^2_{L^{2r}}\,d\tau \\
\lesssim
\Big(\sup_{0\leq\tau\leq t}\langle \tau\rangle^{\frac{d}{4r}+\frac12}\|\nabla\upsilon^{\ell}(\tau)\|_{L^{2r}}\Big)^2
\int_0^t\langle t-\tau\rangle^{-\frac{s_{1}+s}2}\langle\tau\rangle^{-\frac d{2r}-1}\,d\tau \lesssim
 \langle t\rangle^{-\frac {s_{1}+s}{2}}{\cD}_{p}^2(t)
 \end{multline}
and we have  if $t\geq2$,
$$\begin{aligned}
\int_{0}^{t}\langle t-\tau\rangle^{-\frac {s_{1}+s}{2}}\|\nabla \upsilon^{h}(\tau)\|^2_{L^{2r}}\,d\tau=
\Big(\int^{1}_{0}+\int^{t}_{1}\Big)(\cdot\cdot\cdot)d\tau\triangleq J_{1}+J_{2}.
\end{aligned}
$$
Using that $\sigma\leq\frac dp-1$ and interpolation  (see Lemma \ref{lem3.3}), we arrive at
$$
\|\nabla\upsilon^h\|_{L^{2r}}\lesssim
\|\nabla\upsilon^h\|_{\dot{B}^{\sigma}_{p,1}}\lesssim \|\nabla\upsilon^h\|^{\frac{1}{2}}_{\dot{B}^{\frac{d}{p}-2}_{p,1}}
\|\nabla\upsilon^h\|^{\frac{1}{2}}_{\dot{B}^{\frac{d}{p}}_{p,1}},$$
which leads to
\begin{eqnarray}\label{R-E80}
J_{1}&\!\!\lesssim\!\!& \langle t\rangle^{-\frac{s_{1}+s}2} \int_{0}^{1}\|\nabla \upsilon^{h}(\tau)\|^2_{L^{2r}}\,d\tau\lesssim \langle t\rangle^{-\frac{s_{1}+s}2} \cX_{p}^2(1).
\end{eqnarray}
For $J_2,$ we get  for all $-s_1<s\leq\frac d2+1,$
$$
J_{2}\lesssim\int_{1}^{t} \langle t-\tau\rangle^{-\frac{s_{1}+s}2}\langle \tau\rangle^{-\frac{d}{2r}-\frac12}\big(\|\tau^{\frac{d}{4r}+\frac14} \nabla\upsilon^{h}\|_{\dot{B}^{\frac{d}{2r}-1}_{p,1}}\big)^2\,d\tau\lesssim
\langle t\rangle^{-\frac{s_1+s}2}\,\cD_p^2(t).
$$
So finally, if $t\geq2,$
\begin{eqnarray}\label{R-E82}
\int_{0}^{t}\langle t-\tau\rangle^{-\frac {s_{1}+s}{2}}\|\nabla \upsilon^{h}(\tau)\|^2_{L^{2r}}\,d\tau
\lesssim \langle t\rangle^{-\frac {s_{1}+s}{2}}\Big(\cX_{p}^2(t)+{\cD}_{p}^2(t)\Big)\cdotp
\end{eqnarray}
The above inequality also holds true for $t\leq2$ since $\langle t\rangle\thickapprox1$ and
$\langle t-\tau\rangle\thickapprox1$ for $0\leq\tau\leq t\leq 2.$
Therefore, combining with \eqref{R-E78} and \eqref{R-E82}, we end up with
\begin{eqnarray}\label{R-E84}
&&\int_{0}^{t}\langle t-\tau\rangle^{-\frac {s_{1}+s}{2}}\Big\|\frac{Q(\nabla \upsilon, \nabla\upsilon)}{1+a}\Big\|^{\ell}_{\dot{B}^{-s_{1}}_{2,\infty}}\,d\tau
\lesssim\langle t\rangle^{-\frac {s_{1}+s}{2}}\big(\cX_{p}^2(t)+{\cD}_{p}^2(t)\big).
\end{eqnarray}
The term $\tilde K_2(a) \theta^{h}\div\upsilon^{\ell}$ may be treated as $K_2(a)\nabla \theta^{h}$
 and  $\tilde K_2(a) \theta^{\ell}\div\upsilon^{h},$ as for example $a\div\upsilon^h.$
To bound  $\tilde K_2(a)\theta^{h}\div\upsilon^{h},$ one has to proceed slightly
 differently. If $t\geq2$ then we start as usual with
$$\begin{aligned}
\int_{0}^{t}\langle t-\tau\rangle^{-\frac {s_{1}+s}{2}}\|\theta^{h}(\tau)\|_{L^{2r}}\|\div \upsilon^{h}(\tau)\|_{L^{2r}}\,d\tau=
\Big(\int^{1}_{0}+\int^{t}_{1}\Big)(\cdot\cdot\cdot)d\tau\triangleq \tilde{J}_{1}+\tilde{J}_{2}.
\end{aligned}
$$
For $J_1,$  one can write
$$\begin{aligned}
\tilde{J}_{1}&\lesssim \langle t\rangle^{-\frac {s_{1}+s}{2}}\|\theta^{h}\|_{L^2([0,1],L^{2r})}\|\div \upsilon^{h}\|_{L^2([0,1],L^{2r})}\\
&\lesssim \langle t\rangle^{-\frac {s_{1}+s}{2}}\Big(\|\theta\|^{h}_{L^1([0,1],\dot{B}^{\frac{d}{p}}_{p,1})}+\|\theta\|^{h}_{\tilde{L}^\infty([0,1],\dot{B}^{\frac{d}{p}-2}_{p,1})}\Big) \Big(\|\upsilon\|^{h}_{L^1([0,1],\dot{B}^{\frac{d}{p}+1}_{p,1})}+\|\upsilon\|^{h}_{\tilde{L}^\infty([0,1],\dot{B}^{\frac{d}{p}-1}_{p,1})}\Big)\\
&\lesssim  \langle t\rangle^{-\frac{s_{1}+s}2} \cX_{p}^2(1)
\end{aligned}$$
and for all $-s_1<s\leq \frac d2+1$, we have, thanks to \eqref{R-E54},
$$\begin{aligned}
\tilde{J}_{2}&\lesssim  \int_1^t\langle t-\tau\rangle^{-\frac {s_{1}+s}2}\langle \tau\rangle^{-\frac{d}{2r}-\frac12}
\bigl(\tau^{\frac{d}{4r}}\|\theta^{h}(\tau)\|_{L^{2r}}\bigr)\bigl(\tau^{\frac{d}{4r}+\frac {1}{2}}\|\div \upsilon^{h}(\tau)\|_{L^{2r}}\bigr)\,d\tau \\  &\lesssim  \langle t\rangle^{-\frac{s_{1}+s}2} {\cD}_{p}^2(t).
\end{aligned}$$
The case $t\leq2$ is left to the reader.
\medbreak
To bound  the term corresponding to  $I(a)\Delta\theta^h,$ one has to
 consider the cases $2\leq p\leq\frac d2$ and $p>\frac d2$
separately.
If $2\leq p\leq\frac d2$ then we have, denoting $\frac1q\triangleq \frac1p+\frac1{2r}$
and $s_2\triangleq\frac dp+\frac{d}{2r}-\frac d2$
and using the embedding $L^q\hookrightarrow \dot B^{-s_2}_{2,\infty},$
\begin{eqnarray}\label{R-E74}
\|I(a)\Delta\theta^{h}\|^{\ell}_{\dot{B}^{-s_{2}}_{2,\infty}}\lesssim \|I(a)\Delta\theta^{h}\|_{L^{q}}\lesssim \|a\|_{L^{2r}}\|\Delta\theta^{h}\|_{L^p}\lesssim \|a\|_{L^{2r}}\|\theta\|^{h}_{\dot{B}^{\frac dp}_{p,1}}
\end{eqnarray}
By repeating the  procedure leading to \eqref{R-E60}-\eqref{R-E61} and using that $-s_2\leq-s_1,$ we get
$$
\begin{aligned}
\int_{0}^{t}\langle t-\tau\rangle^{-\frac {s_{1}+s}{2}}\|I(a)\Delta\theta^{h}\|^{\ell}_{\dot{B}^{-s_{1}}_{2,\infty}}\,d\tau&\lesssim \int_{0}^{t}\langle t-\tau\rangle^{-\frac {s_{1}+s}{2}}\|I(a)\Delta\theta^{h}\|^{\ell}_{\dot{B}^{-s_{2}}_{2,\infty}}\,d\tau\\
&\lesssim \langle t\rangle^{-\frac {s_{1}+s}{2}}\Big({\cD}_{p}^2(t)+\cX_{p}(t){\cD}_{p}(t)\Big)\cdotp
\end{aligned}
$$
In the case   $d/2< p<d,$  the regularity exponent $\frac dp-2$ is
negative, which precludes $\Delta\theta^h$ to be in any Lebesgue space.
 However, it follows from \eqref{R-E23}, taking $\sigma=2-d/p,$ that
 $$\|I(a)\Delta\theta^h\|_{\dot B^{-s_0}_{2,\infty}}^\ell\leq C \bigl(\|I(a)\|_{L^{p^*}}+\|I(a)\|_{\dot B^{2-\frac {d}{p}}_{p,1}}\bigr)\|\Delta\theta^h\|_{\dot B^{\frac{d}{p}-2}_{p,\infty}}$$
  with $s_0\triangleq\frac{2d}p-\frac d2$ and $\frac1{p^*}=\frac12-\frac1p\cdotp$
  \medbreak
Now, Proposition \ref{prop:compo} and obvious embedding ensure that
$$
\|I(a)\|_{\dot B^{2-\frac {d}{p}}_{p,1}}\lesssim\|a\|_{\dot B^{2-\frac dp}_{p,1}}\lesssim
\|a\|^\ell_{\dot B^{2-s_0}_{2,1}}+\|a\|^h_{\dot B^{\frac dp}_{p,1}}
$$
and,  because $p^*\geq p,$  we have  $\dot B^{\frac dp}_{2,1}\hookrightarrow L^{p^*}$
and $\dot B^{s_0}_{p,1}\hookrightarrow L^{p^*},$ whence
$$
\|I(a)\|_{L^{p^*}}\lesssim\|a\|_{L^{p^*}}\lesssim
\|a\|_{\dot B^{\frac dp}_{2,1}}^\ell+ \|a\|_{\dot B^{\frac dp}_{p,1}}^h.
$$
Therefore we eventually get
\begin{equation}\label{R-E68}
\|I(a)\Delta\theta^{h}\|^{\ell}_{\dot{B}^{-s_{0}}_{2,\infty}}\lesssim (\|a\|_{\dot B^{2-s_0}_{2,1}}^\ell+\|a\|_{\dot B^{\frac dp}_{2,1}}^\ell
+\|a\|_{\dot B^{\frac dp}_{p,1}}^h)\|\theta^h\|_{\dot B^{\frac dp}_{p,1}}.
\end{equation}
If $t\geq2$  then \eqref{R-E68} implies that
\begin{eqnarray}\label{R-E69}
\int_{0}^{t}\langle t-\tau\rangle^{-\frac {s_{1}+s}{2}}\|I(a)\Delta\theta^{h}\|^{\ell}_{\dot{B}^{-s_{0}}_{2,\infty}}\,d\tau &\!\!\!\lesssim\!\!\!&\int_{0}^{t}\langle t-\tau\rangle^{-\frac {s_{1}+s}{2}}\bigl(\|a\|^\ell_{\dot B^{\min(2-s_0,\frac dp)}_{2,1}}\!+\|a\|_{\dot B^{\frac dp}_{p,1}}^h\bigr)\|\theta\|^h_{\dot B^{\frac dp}_{p,1}}d\tau \nonumber\\&\!\!\!=\!\!\!&
\Big(\int^{1}_{0}+\int^{t}_{1}\Big)(\cdot\cdot\cdot)d\tau\triangleq \tilde{I}_{1}+\tilde{I}_{2}.
\end{eqnarray}
On the one hand, the definitions of $\mathcal{D}_{p}$ and $\mathcal{X}_{p}$ ensure that
\begin{eqnarray}\label{R-E70}\tilde{I}_{1}\lesssim  \langle t\rangle^{-\frac {s_{1}+s}{2}}{\cD}_{p}(1)\cX_{p}(1).\end{eqnarray}
On the other hand, thanks to the fact that
\begin{eqnarray}\label{R-E71}
\sup_{\tau\in[1,t]}\tau^\alpha\|\theta^{h}(\tau)\|_{\dot{B}^{\frac{d}{p}}_{p,1}}\lesssim \cD_p(t),
\end{eqnarray}
 that $\alpha\geq\frac{s_1+s}2$ for all $s\leq1+\frac d2$
 and that $2-s_0>-s_1$ for all $s_1$ satisfying \eqref{eq:s1} (if $p>\frac d2$),
 we end up with
$$\begin{aligned}\tilde{I}_{2}&\lesssim
\int_{1}^{t}\langle t-\tau\rangle^{-\frac {s_{1}+s}{2}}\langle\tau\rangle^{-\alpha}
\bigl(\|a^\ell\|_{\dot B^{\min(2-s_0,\frac dp)}_{2,1}}\!+\|a\|_{\dot B^{\frac dp}_{p,1}}^h\bigr)
\bigl(\tau^{\alpha}\|\theta\|^h_{\dot B^{\frac dp}_{p,1}}\bigr)\,d\tau
\\&\lesssim \langle t\rangle^{-\frac {s_{1}+s}{2}}\Big({\cD}_{p}^2(t)+\cX_p(t)\cD_p(t)\Big).
\end{aligned}$$
Putting together with \eqref{R-E69} and using that $-s_0\leq-s_1,$
 we thus get if $t\geq2,$
 $$
 \int_{0}^{t}\langle t-\tau\rangle^{-\frac {s_{1}+s}{2}}\|I(a)\Delta\theta^{h}\|^{\ell}_{\dot{B}^{-s_{1}}_{2,\infty}}\,d\tau \lesssim  \langle t\rangle^{-\frac {s_{1}+s}{2}}\Big({\cD}_{p}^2(t)+\cX_p(t)\cD_p(t)\Big)\cdotp
$$
 That   the above inequality  holds for  $t\leq2$   is a direct consequence of the definitions
of $\cD_p$ and $\cX_p.$
\medbreak
Putting together all the above estimates completes the proof of Inequality \eqref{R-E51}.
Then,   combining with \eqref{R-E49} for bounding the term of \eqref{R-E30} pertaining to the data, we get
\begin{equation}\label{R-E89}
\: \langle t\rangle^{\frac {s_{1}+s}2}\|(a,\upsilon,\theta)(t)\|_{\dot B^s_{2,1}}^\ell \lesssim
{\cD}_{p,0}+{\cD}^2_{p}(t)+\cX^2_{p}(t) \quad\hbox{for all }\, t\geq0\ \hbox{ and }\:
 -s_{1}<s\leq \frac d2+1.
\end{equation}


\subsubsection*{Step 2: Decay estimates for the high frequencies of $(\nabla a, \upsilon, \theta)$}
Let us first recall the following  elementary result (see the proof in \cite{DX}).
\begin{lem}\label{lem4.2}
Let $X:[0,T]\to\R_+$ be a continuous function. Assume  that
$X^p$ is differentiable for some $p\geq1$ and satisfies
$$
\frac{1}{p}\frac{d}{dt}X^{p}+ M X^{p}\leq FX^{p-1}
$$
for some constant $M\geq0$ and measurable function $F:[0,T]\to\R_+.$
\medbreak
Denote  $X_{\varepsilon}=(X^{p}+\varepsilon^p)^{1/p}$ for $\varepsilon>0$. Then it holds that
$$\frac{d}{dt}X_{\varepsilon}+M X_{\varepsilon}\leq F+M\varepsilon.$$
\end{lem}

Let $\cP\triangleq\Id+\nabla(-\Delta)^{-1}\mathrm{div}$ be the Leray projector onto divergence-free vector fields.
It follows from \eqref{R-E9} that  $\cP u$ satisfies the following ordinary heat equation:
$$
\d_t\cP \upsilon -\wt\mu\Delta\cP \upsilon=\cP g.
$$
Applying $\ddj$ to the above equation gives for all $j\in\Z,$
$$
\d_t\cP \upsilon_j -\wt\mu\Delta\cP \upsilon_j=\cP g_j\quad\hbox{with }\ \upsilon_j\triangleq \ddj \upsilon\ \hbox{ and }\
g_j\triangleq \ddj g.
$$
According to Proposition \ref{prop3.4}, we thus  end up  for some constant $c_p>0$ with
$$
\frac1p\frac d{dt}\|\cP \upsilon_j\|_{L^p}^p+c_p\wt\mu2^{2j}\|\cP \upsilon_j\|_{L^p}^p
\leq \|\cP g_j\|_{L^p}\|\cP \upsilon_j\|_{L^p}^{p-1}.
$$
Hence, using the notation $\|\cdot\|_{\varepsilon,\, L^p}\triangleq(\|\cdot\|_{L^p}^p+\varepsilon^p)^{1/p},$
it follows from Lemma \ref{lem4.2} that for all $\ep>0,$
\begin{equation}\label{R-E94}
\frac d{dt}\|\cP \upsilon_j\|_{\varepsilon,\, L^p}+c_p\wt\mu2^{2j}\|\cP \upsilon_j\|_{\varepsilon,\,L^p}\leq \|\cP g_j\|_{L^p}+c_p\wt\mu2^{2j}\varepsilon.
\end{equation}
Next, we observe that $(a,w)$ with $w$ being the   \emph{effective velocity}
 defined in \eqref{eq:w} fulfills
$$
 \left\{\begin{array}{l}
 \d_tw-\Delta w=\nabla(-\Delta)^{-1}(f-\div g)-\gamma\nabla \theta+w-(-\Delta)^{-1}\nabla a,\\[1ex]
 \d_ta+a=f-\div w.\end{array}\right.
$$
Arguing as for $\cP\upsilon,$ one can  arrive at
\begin{multline}\label{R-E99}
\frac d{dt}\|w_j\|_{\varepsilon,\, L^p}+c_{p}2^{2j}\|w_j\|_{\varepsilon,\, L^p}\leq \|\nabla(-\Delta)^{-1}(f_{j}-\div g_{j})\|_{L^p}\\
+\|-\gamma\nabla\theta_{j}+w_{j}-(-\Delta)^{-1}\nabla a_{j}\|_{L^p}+c_{p}2^{2j}\varepsilon,
\end{multline}
and, denoting $R_j^i\triangleq[u\cdot\nabla,\d_i\ddj]a$ for $i=1,\cdots,d,$
\begin{multline}\label{R-E100}
\frac d{dt}\|\nabla a_j\|_{\varepsilon,\, L^p}+\|\nabla a_j\|_{\varepsilon,\, L^p}\leq\Big(\frac{1}{p}\|\div \upsilon\|_{L^\infty}\|\nabla a_j\|_{L^p}+
\|\nabla\dot{\Delta}_{j}(a\div \upsilon)\|_{L^p}\\+C2^{2j}\|w_{j}\|_{L^p}+\|R_{j}\|_{L^p}\Big)+\varepsilon.
\end{multline}
Similarly, as the temperature $\theta$ satisfies
\begin{equation}\label{R-E98}
\d_t\theta-\beta\Delta\theta+\gamma\div w+a=k,
 \end{equation}
we have, according to Proposition \ref{prop3.4},  that
\begin{eqnarray}\label{R-E101}
\frac d{dt}\|\theta_j\|_{\varepsilon,\, L^p}+c_{p}\beta2^{2j}\|\theta_j\|_{\varepsilon,\, L^p}\leq \|k_{j}-\gamma\div w_{j}-a_{j}\|_{L^p}+c_{p}2^{2j}\varepsilon.
\end{eqnarray}
Adding up \eqref{R-E94},  \eqref{R-E99}, $A c_p\times$\eqref{R-E100}
and $B2^{-j}\times$\eqref{R-E101} for some $A,B>0$ (to be chosen afterward)
gives that
$$\displaylines{
\frac{d}{dt}\bigl(\|\cP \upsilon_j\|_{\varepsilon,\, L^p}+\|w_j\|_{\varepsilon,\, L^p}+ A c_p\|\nabla a_j\|_{\varepsilon,\, L^p}+2^{-j}\beta B\|\theta_j\|_{\varepsilon,\, L^p}\bigr)
+c_p2^{2j}\bigl(\wt\mu\|\cP \upsilon_j\|_{\varepsilon,\, L^p}+\|w_j\|_{\varepsilon,\, L^p})\hfill\cr\hfill+ A c_p\|\nabla a_j\|_{\varepsilon,\, L^p}+c_{p}\beta B2^{j}\|\theta_j\|_{\varepsilon,\, L^p}
\leq  \bigl(\|\cP g_j\|_{L^p}+\|\nabla(-\Delta)^{-1}(f_j-\div g_j)\|_{L^p}\bigr)\hfill\cr\hfill
+ A c_p\Bigl(\frac{1}{p}\|\mathrm{div}\upsilon\|_{L^\infty}\|\nabla a_j\|_{L^p}+\|\nabla\ddj(a\div \upsilon)\|_{L^p}+\|R_j\|_{L^p}\Bigr)+2^{-j}B\|k_{j}\|_{L^p}
\hfill\cr\hfill+2^{-j}B\|\gamma\div w_{j}+a_{j}\|_{L^p}
+\|-\gamma\nabla\theta_{j}+w_j-(-\Delta)^{-1}\nabla a_{j}\|_{L^p}+C A c_p2^{2j}\|w_j\|_{L^p}+M_{\varepsilon},}
$$
where $M_{\varepsilon}\triangleq(c_{p}\wt\mu2^{2j}+c_{p}2^{2j}+ A c_{p}+c_{p}B\beta 2^{j})\varepsilon.$
\medbreak
Because  $(-\Delta)^{-1}$ is a homogeneous Fourier multiplier of degree $-2$, we have
\begin{eqnarray}\label{R-E102}
\|(-\Delta)^{-1}\nabla a_{j}\|_{L^p}\lesssim 2^{-2j}\|\nabla a_{j}\|_{L^p}\lesssim 2^{-2j_0}\|\nabla a_{j}\|_{\varepsilon,\, L^p}
\quad\hbox{for }\ j\geq j_0-1.
\end{eqnarray}
Choosing first $B$ large enough, then $A$ suitably small, and finally $ j_{0}$ suitably large, one can  absorb
the last line of the above inequality  by the l.h.s. Hence,  there exists a constant $c_{0}>0$ such that for all $j\geq j_{0}-1$,
$$\displaylines{
\frac{d}{dt}\bigl(\|\cP \upsilon_j\|_{\varepsilon,\, L^p}+\|w_j\|_{\varepsilon,\, L^p}+ A c_p\|\nabla a_j\|_{\varepsilon,\, L^p}+2^{-j}\beta B\|\theta_j\|_{\varepsilon,\, L^p}\bigr)
+c_0\bigl(\|\cP \upsilon_j\|_{\varepsilon,\, L^p}+\|w_j\|_{\varepsilon,\, L^p}\hfill\cr\hfill+A c_p\|\nabla a_j\|_{\varepsilon,\, L^p}+2^{-j}\beta B\|\theta_j\|_{\varepsilon,\, L^p})
\leq  \bigl(\|g_j\|_{L^p}+\|\dot{\Delta}_{j}(a\upsilon)\|_{L^p}\bigr)\hfill\cr\hfill
+A c_p\Bigl(\frac{1}{p}\|\mathrm{div}u\|_{L^\infty}\|\nabla a_j\|_{L^p}+\|\nabla\ddj(a\div u)\|_{L^p}+\|R_j\|_{L^p}\Bigr)+2^{-j}B \|k_{j}\|_{L^p}+M_{\varepsilon}.}
$$
Then, integrating in time and having $\varepsilon$ tend to $0$, we arrive for all $j\geq j_0-1$ at
\begin{equation}\label{R-E103}
e^{c_0t}\|(\cP \upsilon_j,w_j,\nabla a_j, 2^{-j}\theta_j)(t)\|_{L^p}
\lesssim \|(\cP \upsilon_j,w_j,\nabla a_j, 2^{-j}\theta_j)(0)\|_{L^p}+\int_0^te^{c_0\tau}S_j(\tau)\,d\tau
\end{equation}
with $S_j\triangleq S_j^1+\cdots+S_j^6$ and
$$\displaylines{
S_j^1\triangleq  \|\ddj(a\upsilon)\|_{L^p},\quad
S_j^2\triangleq\|g_j\|_{L^p},\quad S_j^3\triangleq 2^{-j}\|k_j\|_{L^p},\cr
S_j^4\triangleq\|\nabla\ddj(a\mathrm{div}\,\upsilon)\|_{L^p},\quad
S_j^5\triangleq \|R_j\|_{L^p}, \quad
S_j^6\triangleq\|\div \upsilon\|_{L^\infty}\|\nabla a_j\|_{L^p}.}$$
It is clear that $(\upsilon_j,\nabla a_j,2^{-j}\theta_j)$ satisfies a similar inequality. Indeed, we have
$$
 \upsilon=w-\nabla(-\Delta)^{-1}a+\cP \upsilon
$$
which leads  for $j\geq j_0-1$  to
$$
\|\upsilon_j\|_{L^p}\lesssim \|w_j\|_{L^p}+\|\cP \upsilon_j\|_{L^p}+ 2^{-2j_0}\|\nabla a_j\|_{L^p}.
$$
Therefore,
 there exists a constant $c_{0}>0$ such that  for all $j\geq j_{0}-1$ and $t\geq0,$
$$
\|(2^ja_j,\upsilon_j, 2^{-j}\theta_j)(t)\|_{L^p}\lesssim e^{-c_{0}t}\|(2^ja_j(0),\upsilon_j(0),2^{-j}\theta_j(0))\|_{L^p}+\int_{0}^{t} e^{-c_{0}(t-\tau)}S_{j}(\tau)\,d\tau.
$$
Now, multiplying both sides by $\langle t\rangle^{\alpha}2^{j(\frac dp-1)},$  taking the supremum on $[0,T],$
and summing up over $j\geq j_0-1$ yields
\begin{multline}\label{R-E106}\|\langle t\rangle^\alpha a\|^h_{\wt L^\infty_T(\dot B^{\frac dp}_{p,1})}+
\|\langle t\rangle^\alpha\upsilon\|^h_{\wt L^\infty_T(\dot B^{\frac dp-1}_{p,1})}+\|\langle t\rangle^\alpha\theta\|^h_{\wt L^\infty_T(\dot B^{\frac dp-2}_{p,1})}\\\lesssim\|a_0\|_{\dot B^{\frac dp}_{p,1}}^h+
\|\upsilon_0\|_{\dot B^{\frac dp-1}_{p,1}}^h+\|\theta_0\|_{\dot B^{\frac dp-2}_{p,1}}^h
\!+\sum_{j\geq j_0-1}\sup_{0\leq t\leq T}\biggl(\langle t\rangle^\alpha\!\int_0^t\!e^{-c_0(t-\tau)}2^{j(\frac dp-1)}S_j(\tau)\,d\tau\biggr)\cdotp
\end{multline}
To treat  the case  $T\leq2,$ we use the fact that
\begin{equation}\label{R-E107}
\sum_{j\geq j_0-1}\sup_{0\leq t\leq 2}\biggl(\langle t\rangle^\alpha\!\int_0^t\!e^{-c_0(t-\tau)}2^{j(\frac dp-1)}S_j(\tau)\,d\tau\biggr)\lesssim \int _{0}^{2}\sum_{j\geq j_0-1}2^{j(\frac dp-1)}S_j(\tau)\,d\tau\cdotp
\end{equation}
The terms in $S_j^1,$ $S_j^4,$ $S_j^5$ and $S_j^6$
as well as those in $S_j^2$ corresponding to $\upsilon\cdot\nabla\upsilon,$
$I(a)\Delta\upsilon$ or $K_1(a)\nabla a$ may be estimated exactly as in \cite{DX}.
Therefore,  it is only a matter of handling the `new'   terms in   $S_j^2,$ and $S_j^3,$ that is \begin{equation}\label{eq:new}
K_2(a)\nabla \theta,\ \theta\nabla K_3(a),\ \upsilon\cdot\nabla\theta,\  I(a)\Delta\theta, \ \frac{Q(\nabla \upsilon, \nabla\upsilon)}{1+a},\ \tilde{K}_{1}(a)\div\upsilon\ \mbox{ and } \ \tilde{K}_{2}(a)\theta\div\upsilon.
\end{equation}
To this end, we shall often use the fact that, by interpolation, we have
\begin{eqnarray}\label{R-E108}
\|a\|_{L^2_{t}(\dot{B}^{\frac{d}{p}}_{p,1})}+\|\upsilon\|_{L^2_{t}(\dot{B}^{\frac{d}{p}}_{p,1})}\lesssim {\cX}_{p}(t).
\end{eqnarray}
For $K_2(a)\nabla \theta$, we still use the  decomposition
$$ K_2(a)\nabla \theta=K_2(a) \nabla \theta^{\ell}+K_2(a) \nabla \theta^{h}.$$
Thanks to H\"{o}lder inequality and Propositions \ref{prop3.1} and \ref{prop:compo}, we get
\begin{eqnarray}\label{R-E109}
&&\|K_2(a)\nabla \theta\|^{h}_{L^1_{t}(\dot{B}^{\frac{d}{p}-1}_{p,1})}\lesssim \|a\|_{L^{2}_{t}(\dot{B}^{\frac{d}{p}}_{p,1})}\|\nabla
\theta^{\ell}\|_{L^2_{t}(\dot{B}^{\frac{d}{p}-1}_{p,1})}+\|a\|_{L^{\infty}_{t}(\dot{B}^{\frac{d}{p}}_{p,1})}\|\nabla
\theta^{h}\|_{L^1_{t}(\dot{B}^{\frac{d}{p}-1}_{p,1})}.
\end{eqnarray}
Now,  interpolation and embedding (recall that $p\geq2$) imply that
$$\|\nabla\theta^{\ell}\|_{L^2_{t}(\dot{B}^{\frac{d}{p}-1}_{p,1})}
\lesssim \|\theta^{\ell}\|^{\frac{1}{2}}_{L^{\infty}_{T}(\dot{B}^{\frac{d}{p}-1}_{p,1})}\|\theta^{\ell}\|^{\frac{1}{2}}_{L^1_{t}(\dot{B}^{\frac{d}{p}+1}_{p,1})}
\lesssim \|\theta\|^{\ell}_{L^{\infty}_{t}(\dot{B}^{\frac{d}{2}-1}_{2,1})}+\|\theta\|^{\ell}_{L^1_{T}(\dot{B}^{\frac{d}{2}+1}_{2,1})}\lesssim {\cX}_{p}(t)
$$ and
$$\|a\|_{L^{\infty}_{t}(\dot{B}^{\frac{d}{p}}_{p,1})}\lesssim \|a\|^{\ell}_{L^{\infty}_{t}(\dot{B}^{\frac{d}{2}-1}_{2,1})}+\|a\|^{h}_{L^{\infty}_{t}(\dot{B}^{\frac{d}{p}}_{p,1})}\lesssim {\cX}_{p}(t).$$
Hence, we arrive at
\begin{eqnarray}\label{R-E110}
\|K_2(a)\nabla \theta\|^{h}_{L^1_{t}(\dot{B}^{\frac{d}{p}-1}_{p,1})}\lesssim {\cX}^2_{p}(t).
\end{eqnarray}
Similarly,  it follows from Proposition \ref{prop3.1} that
$$\begin{aligned}
\|\theta\nabla K_3(a)\|^{h}_{L^1_{t}(\dot{B}^{\frac{d}{p}-1}_{p,1})}&\lesssim\|\theta^{\ell}\|_{L^{2}_{t}(\dot{B}^{\frac{d}{p}}_{p,1})}\|\nabla K_3(a)\|_{L^{2}_{T}(\dot{B}^{\frac{d}{p}-1}_{p,1})}+\|\theta^{h}\|_{L^{1}_{t}(\dot{B}^{\frac{d}{p}}_{p,1})}\|\nabla K_3(a)\|_{L^{\infty}_{T}(\dot{B}^{\frac{d}{p}-1}_{p,1})}\\&\lesssim {\cX}^2_{p}(t).\end{aligned}
$$
Next,  for  $\upsilon\cdot\nabla\theta^\ell$, we
 use the following obvious inequality
 \begin{equation}\label{R-E121b}
\|z\|^h_{\wt L^\infty_T(\dot B^{\frac dq-\varsigma}_{q,1})}\lesssim2^{-j_0\varsigma} \|z\|^h_{\wt L^\infty_T(\dot B^{\frac dq}_{p,1})}\quad\hbox{for }\ q=2,p\ \hbox{ and }
 \varsigma\geq0,
\end{equation}
which, combined with Proposition \ref{prop3.1} implies that
$$\begin{aligned}
\|\upsilon\cdot\nabla\theta^\ell\|^{h}_{L^1_{t}(\dot{B}^{\frac{d}{p}-2}_{p,1})}&\lesssim
2^{-j_0}\|\upsilon\cdot\nabla\theta^\ell\|_{L^1_{t}(\dot{B}^{\frac{d}{p}-1}_{p,1})}\\
&\lesssim \|\upsilon\|_{L^{2}_{t}(\dot{B}^{\frac{d}{p}}_{p,1})}\|\nabla \theta^\ell\|_{L^{2}_{t}(\dot{B}^{\frac{d}{2}-1}_{2,1})},\end{aligned}
$$
and for  $\upsilon\cdot\nabla\theta^h,$ we have
$$
\|\upsilon\cdot\nabla\theta^h\|^{h}_{L^1_{t}(\dot{B}^{\frac{d}{p}-2}_{p,1})}\lesssim
\|\upsilon\|_{L^{2}_{t}(\dot{B}^{\frac{d}{p}}_{p,1})}\|\nabla \theta^h\|_{L^{2}_{t}(\dot{B}^{\frac{d}{p}-2}_{p,1})}.
$$
Therefore we have
\begin{eqnarray}\label{R-E113}
\|\upsilon\cdot\nabla\theta\|^{h}_{L^1_{t}(\dot{B}^{\frac{d}{p}-2}_{p,1})}\lesssim {\cX}^2_{p}(t).
\end{eqnarray}
Similarly,
$$
\|I(a)\Delta\theta\|^{h}_{L^1_{t}(\dot{B}^{\frac{d}{p}-2}_{p,1})}\lesssim\|a\|_{L^\infty_{t}(\dot{B}^{\frac{d}{p}}_{p,1})}\Big(2^{-j_0}\|\Delta\theta^{\ell}\|_{L^1_{T}(\dot{B}^{\frac{d}{2}-1}_{2,1})}
+\|\Delta\theta^{h}\|_{L^1_{T}(\dot{B}^{\frac{d}{p}-2}_{p,1})}\Big)\lesssim {\cX}^2_{p}(t).
$$
Because $p<d$ and $d\geq3$, it follows from Proposition \ref{prop3.1} and \eqref{eq:smalla}  that
\begin{eqnarray}\label{R-E115}
&&\Big\|\frac{Q(\nabla \upsilon, \nabla\upsilon)}{1+a}\Big\|^{h}_{L^1_{t}(\dot{B}^{\frac{d}{p}-2}_{p,1})}\lesssim (1+\|a\|_{L^\infty_{t}(\dot{B}^{\frac{d}{p}}_{p,1})})\|\nabla \upsilon\|^2_{L^2_{t}(\dot{B}^{\frac{d}{p}-1}_{p,1})}\lesssim{\cX}^2_{p}(t)
\end{eqnarray}
and
\begin{eqnarray}\label{R-E116}
\|\tilde{K}_{1}(a)\div\upsilon\|^{h}_{L^1_{t}(\dot{B}^{\frac{d}{p}-2}_{p,1})}\lesssim 2^{-j_{0}}\|a\|_{L^{2}_{T}(\dot{B}^{\frac{d}{p}}_{p,1})}\|\div\upsilon\|_{L^2_{t}(\dot{B}^{\frac{d}{p}-1}_{p,1})}\lesssim {\cX}^2_{p}(t).
\end{eqnarray}
For the  term $\tilde{K}_{2}(a)\theta\div\upsilon$, we use the decomposition
$$ \tilde{K}_{2}(a)\theta\div\upsilon= \tilde{K}_{2}(a)\theta^{\ell}\div\upsilon+ \tilde{K}_{2}(a)\theta^{h}\div\upsilon. $$
Now, we have
\begin{eqnarray}\label{R-E117}
\|\tilde{K}_{2}(a)\theta^{\ell}\div\upsilon\|^{h}_{L^1_{t}(\dot{B}^{\frac{d}{p}-2}_{p,1})}&\!\!\!\lesssim\!\!\! & 2^{-j_{0}}\|\tilde{K}_{2}(a)\theta^{\ell}\div\upsilon\|^{h}_{L^1_{t}(\dot{B}^{\frac{d}{p}-1}_{p,1})}\nonumber\\ &\!\!\!\lesssim\!\!\!&
(1+\|a\|_{L^\infty_{t}(\dot{B}^{\frac{d}{p}}_{p,1})})\|\theta^{\ell}\|_{L^2_{t}(\dot{B}^{\frac{d}{p}}_{p,1})}\|\div\upsilon\|_{L^2_{t}(\dot{B}^{\frac{d}{p}-1}_{p,1})}\lesssim{\cX}^2_{p}(t)
\end{eqnarray}
and
$$
\begin{aligned}
\|\tilde{K}_{2}(a)\theta^{h}\div\upsilon\|^{h}_{L^1_{t}(\dot{B}^{\frac{d}{p}-2}_{p,1})}&\lesssim
 (1+\|a\|_{L^\infty_{t}(\dot{B}^{\frac{d}{p}}_{p,1})})\|\div\upsilon\|_{L^\infty_{t}(\dot{B}^{\frac{d}{p}-2}_{p,1})}\|\theta^{h}\|_{L^1_{t}(\dot{B}^{\frac{d}{p}}_{p,1})}\\&\lesssim
  (\|\upsilon\|^{\ell}_{L^\infty_{t}(\dot{B}^{\frac{d}{2}-1}_{2,1})}+
\|\upsilon\|^{h}_{L^\infty_{t}(\dot{B}^{\frac{d}{p}-1}_{p,1})})\|\theta^{h}\|_{L^1_{t}(\dot{B}^{\frac{d}{p}}_{p,1})}
\lesssim {\cX}^2_{p}(t).
\end{aligned}
$$
Therefore, putting together all the above estimates, we conclude that
\begin{equation}\label{R-E119}
\sum_{j\geq j_0-1}\sup_{0\leq t\leq 2}\biggl(\langle t\rangle^\alpha\!\int_0^t\!e^{-c_0(t-\tau)}2^{j(\frac dp-1)}S_j(\tau)\,d\tau\biggr)\lesssim
{\cX}^2_{p}(2).
\end{equation}
To  bound the supremum for $2\leq t \leq T$ in the last term of \eqref{R-E106},
one can  split the integral on $[0,t]$ into integrals
on  $[0,1]$ and $[1,t].$
The  $[0,1]$  part can be handled  exactly as the supremum on $[0,2]$ treated before.
For  the $[1,t]$ part of the integral,  we use the fact that
\begin{equation}\label{R-E121}
\sum_{j\geq j_0-1}\sup_{2\leq t\leq T}\biggl(\langle t\rangle^\alpha\!\int_1^te^{c_0(\tau-t)}2^{j(\frac dp-1)}S_j(\tau)\,d\tau\biggr)\lesssim
\sum_{j\geq j_0-1} 2^{j(\frac dp-1)}\sup_{1\leq t\leq T} t^\alpha S_j(t).
\end{equation}
In order to bound the term corresponding to $S_j^1,$
we decompose $a$ and $\upsilon$ into low and high frequencies.
Now, we obviously have
$$\begin{aligned}
\|t^\alpha a\upsilon^h\|_{\wt L_t^\infty(\dot B^{\frac dp-1}_{p,1})}
&\lesssim \|a\|_{\wt L_t^\infty(\dot B^{\frac dp}_{p,1})}\|t^\alpha\upsilon
 \|_{\wt L_t^\infty(\dot B^{\frac dp-1}_{p,1})}^h\lesssim \cX_p(t)\cD_p(t)\\
 \|t^\alpha a^h\upsilon\|_{\wt L_t^\infty(\dot B^{\frac dp-1}_{p,1})}
&\lesssim \|t^\alpha a\|^h_{\wt L_t^\infty(\dot B^{\frac dp}_{p,1})}\|\upsilon
 \|_{\wt L_t^\infty(\dot B^{\frac dp-1}_{p,1})}\lesssim \cX_p(t)\cD_p(t),\end{aligned}
 $$
 and, using \eqref{R-E121b},
 $$\begin{aligned}
\|t^\alpha a^\ell\upsilon^\ell\|_{\wt L_t^\infty(\dot B^{\frac dp-1}_{p,1})}^h
&\lesssim \|t^\alpha a^\ell\upsilon^\ell\|_{\wt L_t^\infty(\dot B^{\frac dp+1}_{p,1})}^h\\
&\lesssim \|t^\alpha a^\ell\upsilon^\ell\|_{\wt L_t^\infty(\dot B^{\frac d2+1}_{2,1})}\\
&\lesssim \|t^{\frac12(s_1+\frac d2-\ep)}a\|^\ell_{\wt L_t^\infty(\dot B^{\frac d2}_{2,1})}
 \|t^{\frac12(s_1+\frac d2+\frac12-\ep)}\upsilon\|^\ell_{\wt L_t^\infty(\dot B^{\frac d2+1}_{2,1})}\\
&\hspace{3cm} +  \|t^{\frac12(s_1+\frac d2-\ep)}\upsilon\|^\ell_{\wt L_t^\infty(\dot B^{\frac d2}_{2,1})}
 \|t^{\frac12(s_1+\frac d2+\frac12-\ep)}a\|^\ell_{\wt L_t^\infty(\dot B^{\frac d2+1}_{2,1})}.
 \end{aligned}
$$
 In order to bound the right-hand side, we notice that, from the definition of tilde norms
 and of $\cD_p,$ one can get for $k=0,1,2,$
 \begin{eqnarray}\label{R-E123}
\qquad\|t^{\frac{s_1}2+\frac d4+\frac k2-\frac 12-\frac \varepsilon2}(D^ka^\ell,D^k\upsilon^\ell,D^k\theta^\ell)\|_{\wt L^\infty_T(\dot B^{\frac d2-1}_{2,1})}
&&\!\!\!\!\!\!\!\!\!\lesssim\|t^{\frac{s_1}2+\frac d4+\frac k2-\frac12-\frac \varepsilon2}(a,\upsilon,\theta)\|^\ell_{L_T^\infty(\dot B^{\frac d2-1+k-\ep}_{2,1})}\nonumber\\
&&\!\!\!\!\!\!\!\!\!\leq \cD_p(T). \end{eqnarray}
  So finally, we have
 $$
 \sum_{j\geq j_0-1} 2^{j(\frac dp-1)} \sup_{1\leq t\leq T} t^\alpha S_j^1(t)
 \lesssim \cD_p^2(t).
 $$

The terms corresponding to $\upsilon\cdot\nabla\upsilon^h,$ $\upsilon^h\cdot\nabla\upsilon^\ell,$
$K_1(a)\nabla a^h$ and $I(a)\Delta\upsilon^h$ may be bounded as $a\upsilon^h$
or $a^h\upsilon.$
As regards $\upsilon^\ell\cdot\nabla\upsilon^\ell,$ one
can write using again \eqref{R-E121b} and \eqref{R-E123},
$$\begin{aligned}
\|t^\alpha\upsilon^\ell\cdot\nabla\upsilon^\ell\|_{\wt L_t^\infty(\dot B^{\frac dp-1}_{p,1})}^h
&\lesssim \|t^\alpha\upsilon^\ell\cdot\nabla\upsilon^\ell\|_{\wt L_t^\infty(\dot B^{\frac d2}_{2,1})}^h\\
&\lesssim \|t^{\frac12(s_1+\frac d2-\ep)}\upsilon^\ell\|_{\wt L_t^\infty(\dot B^{\frac d2}_{2,1})}
 \|t^{\frac12(s_1+\frac d2+\frac12-\ep)}\nabla\upsilon^\ell\|_{\wt L_t^\infty(\dot B^{\frac d2}_{2,1})}\\
 &\lesssim \cD_p^2(t)
\end{aligned}
$$
and similarly, thanks to \eqref{R-E44},
$$
\begin{aligned}
\|t^\alpha K_1(a)\nabla a^\ell\|_{\wt L_t^\infty(\dot B^{\frac dp-1}_{p,1})}
&\lesssim  \|t^{\frac12(s_1+\frac d2-\ep)}a\|_{\wt L_t^\infty(\dot B^{\frac dp}_{p,1})}
 \|t^{\frac12(s_1+\frac d2+\frac12-\ep)}\nabla a^\ell\|_{\wt L_t^\infty(\dot B^{\frac d2}_{2,1})}\\
 \|t^\alpha I(a)\Delta\upsilon^\ell\|_{\wt L_t^\infty(\dot B^{\frac dp-1}_{p,1})}
&\lesssim  \|t^{\frac12(s_1+\frac d2-\ep)}a\|_{\wt L_t^\infty(\dot B^{\frac dp}_{p,1})}
 \|t^{\frac12(s_1+\frac d2+\frac12-\ep)}\Delta\upsilon^\ell\|_{\wt L_t^\infty(\dot B^{\frac d2-1}_{2,1})}.
 \end{aligned}
 $$
 The  terms in $S_j^4,$ $S_j^5$ and $S_j^6$ may be treated along the same lines.
 So let us next focus on the terms of $S_j^2$ and $S_j^3$
 corresponding to \eqref{eq:new}.
 Then  Proposition \ref{prop3.1} and Inequality \eqref{R-E121b}  ensure that
$$\begin{aligned}
\|t^\alpha K_2(a)\nabla \theta^{\ell}\|^{h}_{\wt L^\infty_T(\dot B^{\frac dp-1}_{p,1})}&\lesssim \|t^{\frac{s_1}2+\frac d4-\frac \varepsilon2}K_2(a)\|_{\wt L^\infty_T(\dot B^{\frac dp}_{p,1})}
\|t^{\frac{s_1}2+\frac d4+\frac12-\frac \varepsilon2}\nabla \theta^\ell\|_{\wt L^\infty_T(\dot B^{\frac dp}_{p,1})} \\&\!\!\!\!\!\!\!\!\!\!\!\!\!\!\!\!\!\lesssim \Big(\|t^{\frac{s_1}2+\frac d4-\frac \varepsilon2}a\|^\ell_{\wt L^\infty_T(\dot B^{\frac d2}_{2,1})}\!+\!\|t^{\frac{s_1}2+\frac d4-\frac \varepsilon2}a\|^h_{\wt L^\infty_T(\dot B^{\frac dp}_{p,1})}\Big)
\|t^{\frac{s_1}2+\frac d4+\frac12-\frac \varepsilon2}\nabla\theta\|^\ell_{\wt L^\infty_T(\dot B^{\frac {d}{2}}_{2,1})}.
\end{aligned}$$
  Consequently, we arrive at
\begin{eqnarray}\label{R-E125}
\|t^\alpha K_2(a)\nabla \theta^{\ell}\|_{\wt L^\infty_T(\dot B^{\frac dp-1}_{p,1})}\lesssim \cD^2_{p}(T).
\end{eqnarray}
Additionally, it follows from Proposition \ref{prop:compo} that
\begin{equation}\label{R-E128}
\|t^\alpha K_2(a)\nabla \theta^{h}\|_{\wt L^\infty_T(\dot B^{\frac dp-1}_{p,1})}\lesssim
\|a\|_{\wt L^\infty_T(\dot B^{\frac dp}_{p,1})}
\|t^\alpha\nabla\theta\|^h_{\wt L^\infty_T(\dot B^{\frac dp-1}_{p,1})}\lesssim \cX_{p}(T)\cD_{p}(T).
 \end{equation}
To bound $t^{\alpha}\theta\nabla K_3(a)$, we use that  $t^{\alpha}\theta\nabla K_3(a)=t^{\alpha}\theta^{\ell}\nabla K_3(a)+t^{\alpha}\theta^{h}\nabla K_3(a)$. The second term can be estimated as in \eqref{R-E128}.
For  the first term, we write $K_3(a)=K'_3(0)a+\hat{K}_3(a)a$ for some smooth function $\hat{K}_3$ vanishing at $0$. Now, we have thanks to \eqref{R-E123},
$$\begin{aligned}
\|t^{\alpha}\theta^\ell\nabla a^{\ell}\|^{h}_{\wt L^\infty_T(\dot B^{\frac dp-1}_{p,1})}&\lesssim
\|t^{\frac{s_1}2+\frac d4-\frac\varepsilon2}\theta^\ell\|_{\wt L^\infty_T(\dot B^{\frac d2}_{2,1})}\|t^{\frac{s_1}2+\frac d4+\frac12-\frac\varepsilon2} \nabla a^{\ell}\|_{\wt L^\infty_T(\dot B^{\frac d2}_{2,1})}\lesssim \cD^2_{p}(T),\\
 \|t^{\alpha}\theta^\ell\nabla a^{h}\|^{h}_{\wt L^\infty_T(\dot B^{\frac dp-1}_{p,1})}&\lesssim \|\theta^\ell\|_{\wt L^\infty_T(\dot B^{\frac d2-1}_{2,1})}\|t^{\alpha}\nabla a^{h}\|_{\wt L^\infty_T(\dot B^{\frac dp-1}_{p,1})}\lesssim \cX_{p}(T)\cD_{p}(T),
\end{aligned}$$
and, using  Proposition \ref{prop:compo}, the the fact that $\hat K_3(0)=0$ and \eqref{R-E44},
$$\begin{aligned}
\|t^{\alpha}\theta^\ell\nabla (\hat{K}_3(a)a)\|^{h}_{\wt L^\infty_T(\dot B^{\frac dp-1}_{p,1})}
&\lesssim \|t^{\frac{s_1}2+\frac d4-\frac\varepsilon2}\theta^\ell\|_{\wt L^\infty_T(\dot B^{\frac d2}_{2,1})}\|t^{\frac{s_1}2+\frac d4-\frac\varepsilon2} a\|_{\wt L^\infty_T(\dot B^{\frac dp}_{p,1})}\|t^{\frac 12} a\|_{\wt L^\infty_T(\dot B^{\frac dp}_{p,1})}\\&\lesssim\cD^3_{p}(T),\end{aligned}
$$
because    $\frac{s_1}2+\frac d4-\frac\varepsilon2>\frac12$ for sufficiently small $\varepsilon$.
In summary, we get
\begin{eqnarray}\label{R-E129}
\|t^{\alpha}\theta\nabla K_3(a)\|^{h}_{\wt L^\infty_T(\dot B^{\frac dp-1}_{p,1})}\lesssim \cD^2_{p}(T)+\cD^3_{p}(T)+\cX_{p}(T)\cD_{p}(T).
 \end{eqnarray}
 Likewise, Proposition \ref{prop3.1} implies that
\begin{eqnarray}\label{R-E130}
&&\|t^{\alpha}\upsilon\cdot\nabla \theta^h\|^{h}_{\wt L^\infty_T(\dot B^{\frac dp-2}_{p,1})}\lesssim \|\upsilon\|_{\wt L^\infty_T(\dot B^{\frac dp-1}_{p,1})}\|t^\alpha\nabla \theta^h\|_{\wt L^\infty_T(\dot B^{\frac dp-1}_{p,1})}\lesssim \cX_{p}(T)\cD_{p}(T),
 \end{eqnarray}
and it follows from \eqref{R-E121b}, \eqref{R-E123} and \eqref{R-E44} that
$$\begin{aligned}
\|t^{\alpha}\upsilon\cdot\nabla \theta^\ell\|^{h}_{\wt L^\infty_T(\dot B^{\frac dp-2}_{p,1})}
&\lesssim\|t^{\frac{s_1}2+\frac d4-\frac \varepsilon2}\upsilon\|_{\wt L^\infty_T(\dot B^{\frac dp}_{p,1})}
\|t^{\frac{s_1}2+\frac d4+\frac 12-\frac \varepsilon2}\nabla \theta^\ell\|_{\wt L^\infty_T(\dot B^{\frac d2}_{2,1})}\\&\lesssim \cD^2_{p}(T).
\end{aligned}$$
Next,  we have
\begin{eqnarray}\label{R-E132}
&&\|t^{\alpha}I(a)\Delta\theta^h\|^{h}_{\wt L^\infty_T(\dot B^{\frac dp-2}_{p,1})}\lesssim\|a\|_{\wt L^\infty_T(\dot B^{\frac dp}_{p,1})}\|t^\alpha\theta\|^h_{\wt L^\infty_T(\dot B^{\frac dp}_{p,1})}\lesssim \cX_{p}(T)\cD_{p}(T)
\end{eqnarray}
and, using \eqref{R-E44},
\begin{equation}\label{R-E133}
\|t^{\alpha}I(a)\Delta\theta^\ell\|^{h}_{\wt L^\infty_T(\dot B^{\frac dp-2}_{p,1})}\lesssim \|t^{\frac{s_1}2+\frac d4+\frac 12-\frac \varepsilon2}\Delta\theta^\ell\|_{\wt L^\infty_T(\dot B^{\frac d2-1}_{2,1})}\|t^{\frac{s_1}2+\frac d4-\frac \varepsilon2}a\|^{h}_{\wt L^\infty_T(\dot B^{\frac dp}_{p,1})}\lesssim \cD^2_{p}(T).
\end{equation}
To bound the term containing $\frac{Q(\nabla \upsilon, \nabla\upsilon)}{1+a}$, we just write that,
owing to Propositions \ref{prop3.1}, \ref{prop:compo}, and to Condition \eqref{eq:smalla},
$$
\displaylines{\Big\|t^{\alpha}\frac{Q(\nabla \upsilon,\!\nabla\upsilon)}{1+a}\Big\|^{h}_{\wt L^\infty_T(\dot B^{\frac dp-2}_{p,1})}\lesssim
\Bigl(\|t^{\frac{s_{1}}{2}+\frac d4+\frac 14-\frac \varepsilon2}\nabla \upsilon\|_{\wt L^\infty_T(\dot B^{\frac dp}_{p,1})}^\ell
\!+\|t^{\frac{s_{1}}{2}+\frac d4+\frac 14-\frac \varepsilon2}\nabla\upsilon\|_{\wt L^\infty_T(\dot B^{\frac dp}_{p,1})}^h\Bigr)^2\cdotp}
$$
Hence, from \eqref{eq:smalla}, \eqref{R-E123} and the definition of $\cD_{p}$, we infer that
\begin{eqnarray}\label{R-E136}
\Big\|t^{\alpha}\frac{Q(\nabla \upsilon, \nabla\upsilon)}{1+a}\Big\|^{h}_{\wt L^\infty_T(\dot B^{\frac dp-2}_{p,1})}\lesssim \cD^2_{p}(T).
\end{eqnarray}
For the next term in \eqref{eq:new}, we start with
the decomposition $$\tilde{K}_{1}(a)\div\upsilon\triangleq \tilde{K}_{1}(a)\div\upsilon^{h}+\tilde{K}_{1}(a)\div\upsilon^{\ell}.$$ Now, by virtue of  \eqref{R-E121b},
Propositions \ref{prop3.1} and \ref{prop:compo},
\begin{eqnarray}\label{R-E137}
&&\|t^{\alpha}\tilde{K}_{1}(a)\div\upsilon^{h}\|^{h}_{\wt L^\infty_T(\dot B^{\frac dp-2}_{p,1})}\lesssim \|a\|_{\wt L^\infty_T(\dot B^{\frac dp}_{p,1})}\|t^{\alpha}\upsilon\|^{h}_{\wt L^\infty_T(\dot B^{\frac dp-1}_{p,1})}\lesssim \cX_{p}(T)\cD_{p}(T)
\end{eqnarray}
and,  thanks to \eqref{R-E44}, \eqref{R-E121b} and \eqref{R-E123},
$$
\|t^{\alpha}\tilde{K}_{1}(a)\div\upsilon^{\ell}\|^{h}_{\wt L^\infty_T(\dot B^{\frac dp-2}_{p,1})}\lesssim
 \|t^{\frac{s_1}2+\frac d4+\frac 12-\frac \varepsilon2}\div\upsilon^\ell\|_{\wt L^\infty_T(\dot B^{\frac d2}_{2,1})}\|t^{\frac{s_1}2+\frac d4-\frac \varepsilon2}a\|_{\wt L^\infty_T(\dot B^{\frac dp}_{p,1})}\lesssim \cD^2_{p}(T).
$$
For $\tilde{K}_{2}(a)\theta\div\upsilon,$  as the regularity of $\theta^h$ is lower than that of $\upsilon^h,$  we use the decomposition $\tilde{K}_{2}(a)\theta\div\upsilon= \tilde{K}_{2}(a)\theta^{\ell}\div\upsilon+ \tilde{K}_{2}(a)\theta^{h}\div\upsilon$ again. Remembering \eqref{eq:smalla}, it follows
from Propositions \ref{prop3.1} and \ref{prop:compo} and Inequality \eqref{R-E123} that
$$\begin{aligned}
\|t^{\alpha} \tilde{K}_{2}(a)\theta^{h}\div\upsilon\|^{h}_{\wt L^\infty_T(\dot B^{\frac dp-2}_{p,1})}
&\lesssim\|t^{\alpha}\theta^{h}\|_{\wt L^\infty_T(\dot B^{\frac dp}_{p,1})}\|\div\upsilon\|_{\wt L^\infty_T(\dot B^{\frac dp-2}_{p,1})}\\&\lesssim\cD_{p}(T)\cX_{p}(T)\\
\|t^{\alpha} \tilde{K}_{2}(a)\theta^{\ell}\div\upsilon\|^{h}_{\wt L^\infty_T(\dot B^{\frac dp-2}_{p,1})}
&\lesssim \|t^{\frac{s_1}2+\frac d4-\frac \varepsilon2}\theta\|^\ell_{\wt L^\infty_T(\dot B^{\frac d2}_{2,1})}\|t^{\frac{s_1}2+\frac d4+\frac 12-\frac \varepsilon2}\div\upsilon\|_{\wt L^\infty_T(\dot B^{\frac dp}_{p,1})}\\&\lesssim\cD^2_{p}(T).\end{aligned}
$$
Putting all above inequalities  together, the r.h.s. of (\ref{R-E121}) can be estimated as follows:
$$
\sum_{j\geq j_0-1} 2^{j(\frac dp-1)}\sup_{1\leq t\leq T} t^\alpha S_j(t)
\lesssim\cX^2_{p}(T)+\cD^3_{p}(T)+\cD^2_{p}(T)+\cD_{p}(T)\cX_{p}(T).
$$
Consequently, keeping in mind that $\cD_p$ is small,  we obtain
\begin{multline}\label{R-E142}
\|\langle t\rangle^\alpha(\nabla a,\upsilon)\|^h_{\wt L^\infty_T(\dot B^{\frac dp-1}_{p,1})}+\|\langle t\rangle^\alpha\theta\|^h_{\wt L^\infty_T(\dot B^{\frac dp-2}_{p,1})}\\\lesssim
\|(\nabla a_0,\upsilon_0)\|_{\dot B^{\frac dp-1}_{p,1}}^h+\|\theta_0\|_{\dot B^{\frac dp-2}_{p,1}}^h
\!+\cX^2_{p}(T)+\cD^2_{p}(T)\cdotp
\end{multline}


\subsubsection*{Step 3: Decay estimates with gain of regularity  for the high frequencies of $\upsilon$ and $\theta$}

We here want to  prove that the parabolic smoothing effect provided
by the last two equations of \eqref{R-E9} allows to get gain of regularity
 and  decay altogether for   $\upsilon$ and $\theta.$
Let us focus on the equation for $\theta$ (handling $\upsilon$ being similar).
Recall that
$$
\d_t\theta-\beta\Delta\theta=-\gamma \div v+k.
$$
Hence
\begin{equation}\label{R-E143}
\d_t(t^{\alpha}\Delta\theta)-\beta\Delta(t^{\alpha}\Delta\theta)=\alpha\beta t^{\alpha-1}\Delta\theta
+\beta t^{\alpha}\Delta (k-\gamma\div v),\qquad
t^\alpha\Delta\theta|_{t=0}=0.
\end{equation}
We thus deduce from Proposition \ref{prop3.6} that
\begin{equation}\label{R-E144}
\|\tau^{\alpha}\nabla^2\theta\|_{\wt L_t^\infty(\dot B^{\frac dp-2}_{p,1})}^h \lesssim \|\tau^{\alpha-1}\Delta \theta\|_{\wt L^\infty_t(\dot B^{\frac dp-4}_{p,1})}^h +\|\tau^{\alpha}\Delta(k-\gamma\div v)\|_{\wt L^\infty_t(\dot B^{\frac dp-4}_{p,1})}^h,
\end{equation}
hence,
\begin{eqnarray}\label{R-E145}
&&\|\tau^{\alpha}\nabla\theta\|_{\wt L_t^\infty(\dot B^{\frac dp-1}_{p,1})}^h \lesssim  \|\tau^{\alpha-1}\theta\|_{\wt L^\infty_t(\dot B^{\frac dp-2}_{p,1})}^h
+\|\tau^{\alpha} v\|_{\wt L^\infty_t(\dot B^{\frac dp-1}_{p,1})}^h
+\|\tau^{\alpha} k\|_{\wt L^\infty_t(\dot B^{\frac dp-2}_{p,1})}^h.
\end{eqnarray}
Because $\alpha\geq1,$ we have
\begin{equation}\label{R-E146}
\|\tau^{\alpha-1}\theta\|_{\wt L^\infty_t(\dot B^{\frac dp-2}_{p,1})}^h \lesssim \|\langle\tau\rangle^{\alpha}\theta\|_{\wt L^\infty_t(\dot B^{\frac dp-2}_{p,1})}^h\quad\!\!\hbox{and}\!\!\quad
\|\tau^{\alpha} \upsilon\|_{\wt L^\infty_t(\dot B^{\frac dp-1}_{p,1})}^h \lesssim \|\langle\tau\rangle^{\alpha}\upsilon\|_{\wt L^\infty_t(\dot B^{\frac dp-1}_{p,1})}^h.
\end{equation}
Bounding   $\|\tau^{\alpha} k\|_{\wt L^\infty_t(\dot B^{\frac dp-2}_{p,1})}^h$
as in Step 2, one can thus conclude that
\begin{equation}\label{R-E148}
\|\tau^{\alpha}\nabla\theta\|_{\wt L_t^\infty(\dot B^{\frac dp-1}_{p,1})}^h \lesssim
\|\langle\tau\rangle^{\alpha}\upsilon\|_{\wt L^\infty_t(\dot B^{\frac dp-1}_{p,1})}^h+\cD^2_{p}(t)+\cX_{p}^2(t).
\end{equation}
Arguing similarly with the second equation of \eqref{R-E9},
we get
\begin{equation}\label{R-E149}
\|\tau^\alpha\nabla \upsilon\|_{\wt L_t^\infty(\dot B^{\frac dp}_{p,1})}^h \lesssim \|\langle\tau\rangle^{\alpha}(a,\theta)\|_{\wt L^\infty_t(\dot B^{\frac dp}_{p,1})}^h+\cD^2_{p}(t)+\cX_{p}^2(t).
\end{equation}

Finally,  bounding the first terms on the right-side of  \eqref{R-E148}-\eqref{R-E149} according
to \eqref{R-E142},  and adding up \eqref{R-E148}-\eqref{R-E149} to \eqref{R-E89} and \eqref{R-E142} yields for all $T\geq0,$
$$
\cD_{p}(T)\lesssim \cD_{p,0}+\|(a_0,\upsilon_0,\theta_0)\|^{\ell}_{\dot B^{\frac d2-1}_{2,1}}+\|(\nabla a_0,\upsilon_0)\|_{\dot B^{\frac dp-1}_{p,1}}^h+\|\theta_0\|_{\dot B^{\frac dp-2}_{p,1}}^h+\cX^2_{p}(T)+\cD^2_{p}(T).
$$
As Theorem \ref{thm1.1} ensures that $\cX_{p}\lesssim \cX_{p,0}$ and as
 $$\|(a_0,u_0,\theta_0)\|^{\ell}_{\dot B^{\frac d2-1}_{2,1}}\lesssim\|(a_0,u_0,\theta_0)\|^{\ell}_{\dot B^{-s_1}_{2,\infty}},$$
one can  conclude
that \eqref{R-E11} is fulfilled for all time if $\cD_{p,0}$ and $\cX_{p,0}$
are small enough. This completes the proof of Theorem \ref{thm2.1}.\qed
\bigbreak
  Corollary \ref{cor2.1} easily follows from Theorem \ref{thm2.1} : let us just
  show the inequality for $\theta$ as an example.
  Since the embedding $\dot{B}^{s}_{2,1}\hookrightarrow\dot{B}^{s-d(1/2-1/p)}_{p,1} \hookrightarrow \dot{B}^{s}_{p,1}$  holds for the low frequencies whenever $p\geq2,$
  we have for all $-s_1<s\leq\frac dp-2,$
  $$
\sup_{t\in[0,T]} \langle t\rangle^{\frac {s_{1}+s}2}\|\Lambda^s \theta\|_{\dot B^0_{p,1}}\lesssim
 \|\langle t\rangle^{\frac {s_{1}+s}2}\theta\|_{L^\infty_T(\dot B^s_{2,1})}^\ell
 +  \|\langle t\rangle^{\frac {s_{1}+s}2}\theta\|_{L^\infty_T(\dot B^s_{p,1})}^h.
$$
  Hence, using \eqref{R-E11} yields
$$
\|\langle t\rangle^{\frac {s_{1}+s}2}\theta\|_{L^\infty_T(\dot B^s_{2,1})}^\ell\lesssim \bigl({\cD}_{p,0}+\|(\nabla a_0,\upsilon_0)\|^{h}_{\dot B^{\frac dp-1}_{p,1}}+\|\theta_0\|_{\dot B^{\frac dp-2}_{p,1}}^h\bigr).
$$
Now, the fact that  $\alpha\geq\frac{s_1+s}2$ for all $s\leq \frac dp-2$ allows to write that
$$
\|\langle t\rangle^{\frac {s_{1}+s}2}\theta\|_{L^\infty_T(\dot B^s_{p,1})}^h
  \lesssim \bigl({\cD}_{p,0}+\|(\nabla a_0,\upsilon_0)\|^{h}_{\dot B^{\frac dp-1}_{p,1}}+\|\theta_0\|_{\dot B^{\frac dp-2}_{p,1}}^h\bigr),
$$
which completes the proof of Corollary \ref{cor2.1}.\qed
\bigbreak
We end this section with some heuristics concerning  the optimality of
the regularity and decay  exponents
in the definition of $\cD_p.$
Let us first explain why the regularity exponent $s$ in  $\cD_{p,1}$
has to satisfy $s\leq\frac d2+1.$ The general fact (based on Inequality \eqref{R-E29})
 that we used repeatedly  is that the time decay  exponent $\delta$
for $\|(f,g,k)\|_{\dot B^{-s_1}_{2,\infty}}$
 must satisfy $\delta\geq \frac{s_1+s}2\cdotp$
Now, if we look  at  the term $a^\ell\,\div u^\ell,$ then
 a necessary condition for having
$$
\|a^\ell\,\div u^\ell\|_{\dot B^{-s_1}_{2,\infty}}^\ell\lesssim
\|a^\ell\|_{\dot B^{\sigma_1}_{2,\infty}}\|\div u^\ell\|_{\dot B^{\sigma_2}_{2,\infty}}
$$
is that  $\sigma_1+\sigma_2\leq \frac d2-s_1.$ As the decay exponent  of the right-hand side is
$s_1+\frac{\sigma_1+\sigma_2+1}2,$ we deduce that  $\delta\leq\frac d4+\frac{s_1}2+\frac12\cdotp$
Hence we must have $s\leq \frac d2+1.$
\smallbreak
To see  that the decay rate in $\cD_{p,2}$
cannot be more than $s_1+\frac{d}2+\frac12,$
one can observe that, owing to $s\leq \frac d2+1,$
the term  $a^\ell\nabla a^\ell$ (which at most has the same regularity as
 $\nabla a^\ell$) cannot be estimated in a space with higher
 regularity than  $\dot B^{\frac d2}_{2,1}.$ As the corresponding
 estimate reads
 $$
 \|a^\ell\nabla a^\ell\|_{\dot B^{\frac d2}_{2,1}}\lesssim \|a^\ell\|_{\dot B^{\frac d2}_{2,1}}
 \|\nabla a^\ell\|_{\dot B^{\frac d2}_{2,1}},$$
the definition of $\cD_{p,1}$ ensures that  the right-hand  side has decay exponent
 $s_1+ \frac d2+\frac 12\cdotp$
A similar argument shows that the decay rate  in the definition of  $\cD_{p,3}$ is optimal.

\end{document}